\documentclass[11pt]{article}
\usepackage[T1]{fontenc}
\usepackage[latin1]{inputenc}
\usepackage{amsmath,amssymb}
\usepackage{graphics}

\newtheorem{thm}{Theorem}[section]
\newtheorem{lem}[thm]{Lemma}

\newtheorem{nota}[thm]{Notation}
\newtheorem{defi}[thm]{Definition}
\newtheorem{prop}[thm]{Proposition}
\newtheorem{rema}[thm]{Remark}

\newenvironment{proof}{\noindent {\bf Proof \phantom{9}}}
{\hfill $\square$ \vspace{0.20cm}}

\def\theequation{\thesection.\arabic{equation}}
\def\be{\begin{eqnarray}}
\def\ee{\end{eqnarray}}
\def\ben{\begin{eqnarray*}}
\def\een{\end{eqnarray*}}
\def\numero{\refstepcounter{equation} (\theequation)}
\def\ala{\nonumber\\}

\let\text=\textstyle

\newcommand{\sm}{{s-}}

\newcommand{\intot}{\displaystyle \int _0^t }

\newcommand{\indiq}{{\bf 1}}

\newcommand{\rit}{\mathbb{R}}
\newcommand{\nit}{\mathbb{N}}
\newcommand{\dit}{\mathbb{D}}

\def\me{\medskip \noindent}
\def\bi{\bigskip \noindent}

\title{\bf A host-parasite multilevel  interacting  process
and continuous approximations}

\author{Sylvie M\'el\'eard$^{1}$ and Sylvie R\oe lly$^{2,3}$}

\date{\today}

\begin{document}

\maketitle

\begin{center}
 \makeatletter\renewcommand{\@makefnmark}{\mbox{$^{\@thefnmark}$}}\makeatother
 \footnotetext[1]{Centre de
  Math\'ematiques Appliqu\'ees, UMR C.N.R.S. 7641, \'Ecole
  Polytechnique, 91128 Palaiseau Cedex, France. Email: sylvie.meleard@polytechnique.edu}
\footnotetext[2]{ Institut f\"ur Mathematik der Universit\"at Potsdam, Am Neuen Palais 10,
14469 Potsdam, Germany. Email: roelly@math.uni-potsdam.de}
\footnotetext[3]{On leave of absence Centre de
  Math\'ematiques Appliqu\'ees, UMR C.N.R.S. 7641, \'Ecole
  Polytechnique, 91128 Palaiseau Cedex, France.}
 \makeatletter\renewcommand{\@makefnmark}{}\makeatother
 \makeatletter\renewcommand{\@makefnmark}{\mbox{$^{\@thefnmark}$}}\makeatother
\end{center}

\begin{abstract}
 We are interested in modeling some two-level
 population dynamics,
 resulting from the interplay of ecological interactions and phenotypic variation of
 individuals (or hosts) and the
 evolution of  cells (or parasites) of two types  living in these
 individuals. The ecological parameters of the
 individual dynamics depend on the number of cells of each type contained by the individual and the cell dynamics depends 
on the trait of the invaded
 individual. 
 
 Our models are rooted in the
 microscopic description of a random (discrete) population of 
 individuals characterized by one or several adaptive traits and cells characterized by their type. 
The population is modeled as a stochastic point process whose generator
 captures the probabilistic dynamics over continuous time of birth,
 mutation and death for individuals and birth and death for
 cells. 
The interaction between individuals (resp. between cells) is described by a competition between individual traits
 (resp. between cell types).
 We look for tractable large population approximations. By combining various
 scalings on population size, birth and death rates and  mutation step, the single microscopic model is shown to lead to contrasting nonlinear
 macroscopic limits of different nature: deterministic approximations, in the form of ordinary, integro- or partial
 differential equations, or probabilistic ones, like stochastic partial differential equations or superprocesses.
 The study of the long time behavior of these processes  seems very hard and we only develop some simple cases enlightening the difficulties involved.  \end{abstract}

\bigskip


% \pagebreak

%%%%%%%%%%%%%%%%%%%%%%%%%%%%%%%%%%%%%%%%%%%%%%%%%%%%%%%%%%%%%%%%%
\section{Introduction}
\label{sec:intro}
%%%%%%%%%%%%%%%%%%%%%%%%%%%%%%%%%%%%%%%%%%%%%%%%%%%%%%%%%%%%%%%%%

In this paper, we are interested in describing the adaptive
effects in a host-parasite system. We model two-level
 population dynamics 
 resulting from the interplay of ecological interactions and phenotypic variation of
 individuals (or hosts) and the
 evolution of  cells (or parasites) of two types  living in these
 individuals. In one hand, the ecological parameters of the
 individual dynamics depend on its number of cells of each type.
 In another hand, the cells develop their own birth and death dynamics and
 their ecological parameters depend on the trait of the
invaded  individual. 

\me We consider more precisely  the following two-level population. The first  level is composed of 
individuals governed by a mutation-birth and death process. Moreover each
individual is a collection of cells of two types (types 1
and 2) which have their own dynamics and compose the second level. The model  can easily be generalized to cells with a
finite number of different types. We denote by $n_1^i$ (resp. by
$n_2^i$) the number of cells of type 1 (resp. of type 2)
living in the individual $i$.  This individual is moreover characterized 
by a continuous quantitative phenotypic trait $x^i$.
The individual $\ i\ $ can be removed or copied according a
birth-and-death process depending on $x^i, n_1^i, n_2^i$.
An offspring usually inherits the trait values of her progenitor except when  a  mutation  occurs during the reproduction
mechanism. In this case the offspring makes an instantaneous mutation step at birth to new trait value.  The death of an individual can be natural or can be due to the competition exerted by the other individuals, for example sharing food.  This competition between individuals through their traits will induce a nonlinear convolution term. 
The cells in the individual $i$ also reproduce and remove
according to another birth-and-death process, depending on $x^i,
n_1^i, n_2^i$. At this second level, cell competition occurs and its pressure depends on the invaded individual trait.

\me Our model generalizes  some approach firstly developed in Dawson
and Hochberg \cite{Dawson:91} and in Wu \cite{Wu:93},
\cite{Wu:95}. In these papers,  a two-level system is studied:
individuals and cells follow a branching dynamics but there is no interaction  between individuals and between cells. Thus all the specific techniques these authors use - as Laplace transforms - are no more available for our model.

\me In this paper, we firstly rigorously construct the underlying mathematical
model and  prove its existence. Thus we obtain moment and martingale properties which are the key point to deduce
 approximations for large individual and cell populations. 
  By combining various scalings on population size, birth and death rates and  mutation step, the single microscopic model is shown to lead to contrasting macroscopic limits
 of different nature: deterministic approximations, in the form of ordinary, integro- or partial
 differential equations, or probabilistic ones, like stochastic partial
 differential equations or superprocesses.
 The study of the long time behavior  of these processes seems very hard and we only develop some simple cases enlightening the difficulties involved.  

%%%%%%%%%%%%%%%%%%%%%%%%%%%%%%%%%%%%%%%%%%%%%%%%%%%%%%%%%%%%%%%%%%%
\section{Population point process}
\label{sec:ppp}
%%%%%%%%%%%%%%%%%%%%%%%%%%%%%%%%%%%%%%%%%%%%%%%%%%%%%%%%%%%%%%%%%%%
\subsection{The model}
%%%%%%%%%%%%%%%%%%%%%%%%%%%%%%%%%%%%%%%%%%%%%%%%%%%%%%%%%%%%%%

We model the evolving population by a stochastic interacting
individual system, where each individual $i$ is characterized by a
vector phenotypic trait value $x^i $ and by the number of its cells of
type 1, $n_1^i$, and of type 2, $n_2^i$. The trait space ${\cal X}$ is assumed to be
a compact subset of $\rit^d$, for some $d\geq 1$.
We denote by $M_F=M_F({\cal X}\times \nit\times \nit)$ the
set of finite non-negative measures on ${\cal X}\times \nit\times
\nit$, endowed with the weak topology.
Let also ${\cal M}$ be the subset of $M_F$ consisting of
all finite point measures:
\begin{equation*}
 {\cal M} = \left\{ \sum_{i=1}^I \delta_{(x^i, n_1^i, n_2^i)} , \; 
   x^i \in {\cal X}, (n_1^i, n_2^i)\in  \nit\times \nit,\
    1\leq i\leq I ,\;  I \in \nit \right\}.
\end{equation*}
Here and below $\delta_{(x,n_1,n_2)}$ denotes the Dirac mass at
$(x,n_1,n_2)$. In case where $I=0$, the measure is the null
measure.

\noindent Therefore, for a population modelled by $\nu=\sum_{i=1}^{I}\delta_{(x^i,n_1^i,n_2^i)}$, the
total number of its individuals is $\langle \nu,1\rangle=I$  and, if we denote by $n:=n_1+n_2$ the number of cells of an individual (irrespective of  type), then  $\langle
\nu,n\rangle=\sum_{i=1}^I(n_1^i+n_2^i)$ is the total number
of cells in the population $\nu$.\\

Let us now describe the two-level dynamics.
Any individual of the population with trait $x$ and cell state $(n_1,n_2)$  
follows a mutation-selection-birth-and-death dynamics with 
\begin{itemize}
   \item    birth (or reproduction) rate $B(x, n_1,n_2)$,
   \item  the reproduction is clonal with probability $1-p(x)$ (the offspring inherits the trait $x$),
   \item a mutation occurs with probability $p(x)$,
   \item  the mutant trait $x+z$ is distributed according the mutation kernel $M(x,z)\, dz$ which only weights $z$ such that $x+z\in {\cal X},$
      \item  death rate $D(x, n_1,n_2)+ \alpha(x,
   n_1,n_2)\sum_{j=1}^IU(x-x^j).$
\end{itemize}
   Thus the  interaction
    between individuals is modeled by a comparison between
    their respective trait values  described by the competition kernel $U$.
    %The functions $B$, $p$, $M$, $D$, $\alpha$, $U$ are  continuous functions. 
By simplicity,  the mutations parameters $p$ and $M$  are assumed to  be  only influenced by the trait $x$. They could also  depend on the cell composition $(n_{1}, n_{2})$ without inducing any additional technical difficulty. \\

\noindent Any cell of type $1$ (resp. of type $2$) inside an individual with trait $x$ and cell state $(n_1,n_2)$
follows a birth-and-death dynamics with 
\begin{itemize}
    \item  birth rate $b_1(x)$, (resp. $b_2(x)$),
    \item  death rate $d_1(x)+\beta_1(x)(n_1 \lambda_{11}+ n_2 \lambda_{12})$, (resp. $d_2(x)+\beta_2(x)(n_1 \lambda_{21}+ n_2
    \lambda_{22})$). 
\end{itemize}
The nonnegative parameters $\lambda_{11}$, $\lambda_{22}$, $\lambda_{12}$, $\lambda_{21}$  quantify the cell interactions. The rate functions $b_1$, $b_2$, $d_1$,  $d_2$, $\beta_1$, $\beta_2$
     are assumed to be continuous (and thus bounded on the compact set ${\cal X}$).\\

\noindent The population dynamics can be described by its possible transitions from a state $\nu$ to the following other states:\\

1 - Individual dynamics due to an individual with trait $x$ and cell state $(n_1,n_2)$:
\be
\nu &\mapsto & \nu +\delta_{(x, n_1,n_2)}  \mbox{ with rate } B(x,n_1,n_2)(1-p(x))\ ;\ala 
\nu &\mapsto & \nu -\delta_{(x,n_1,n_2)} \mbox{ with rate } D(x, n_1,n_2)+\alpha(x,
n_1,n_2)\sum_{j=1}^I U(x-x^j)\ ;\ala 
\nu &\mapsto& \nu +\delta_{(x+z,n_1,n_2)}  \mbox{ with rate }
B(x,n_1,n_2)\,p(x),  \mbox{ where } z \mbox{ is
distributed following } M(x,z)\, dz.\nonumber
\ee

2 - Cell dynamics:
\be
\nu &\mapsto & \nu +\delta_{(x, n_1+1,n_2)} -\delta_{(x, n_1,n_2)} \mbox{ with rate } b_1(x)\ ;\ala 
\nu &\mapsto & \nu +\delta_{(x, n_1,n_2+1)}-\delta_{(x, n_1,n_2)}
 \mbox{ with rate } b_2(x)\ ;\ala
 \nu &\mapsto& \nu +\delta_{(x, n_1-1,n_2)} -\delta_{(x, n_1,n_2)} \mbox{ with rate }
 d_1(x)+\beta_1(x)( \lambda_{11}n_1+ \lambda_{12}n_2 )\ ;\ala
\nu &\mapsto& \nu +\delta_{(x, n_1,n_2-1)} -\delta_{(x, n_1,n_2)}
 \mbox{ with rate } d_2(x)+\beta_2(x)( \lambda_{21}n_1+  \lambda_{22}n_2).
\nonumber
\ee

\noindent Let us now prove the existence of a c\`adl\`ag Markov process $(\nu_t)_{t\geq 0}$
belonging to $\dit(\rit_+,{\cal M})$ 
modeling the dynamics of such a discrete population. More precisely, we consider
\begin{equation}
 \label{pop}
 \nu_t = \sum_{i=1}^{I(t)} \delta_{(X^i(t),N_1^i(t), N_2^i(t))}
\end{equation}
where $I(t) \in {\mathbb{N}}$ stands for the number of individuals
alive at time $t$, $X^1(t),...,X^{I(t)}(t) \in {\cal X}$ describes the
 traits of these individuals at time $t$ and
$N_1^1(t),...,N_1^{I(t)}(t)$  (resp. $N_2^1(t),...,N_2^{I(t)}(t))$
are the  numbers of cells of type 1 (resp. of type 2) for the individuals alive at time $t$.\\
To write down the infinitesimal generator of $\nu$, we need an appropriate class of 
test functions. For bounded measurable functions $\phi$,
$f$, $g_1$, $g_2$ defined respectively on $\rit$, $\rit^d$,
$\nit$ and $\nit$, $\phi_{fg_1g_2}$ is given by 
\be
\label{phi}
\phi_{fg_1g_2}(\nu)&:=&\phi(<\nu,fg_1g_2>)=\phi\left(\int_{{\cal
X}\times \nit^2 } f(x)g_1(n_1)g_2(n_2)\nu(dx,dn_1,dn_2)\right)\ala
&=&\phi\left(\sum_{n_1,n_2\in \nit^2}\int_{\cal X}
f(x)g_1(n_1)g_2(n_2)\nu(n_1,n_2,dx)\right). 
\ee

\noindent The infinitesimal generator $L$ of the Markov process $(\nu_t, t\geq 0)$ applied to such
function $\phi_{fg_1g_2}$ is given by:

\begin{align}
\label{generator} &L\phi_{fg_1g_2}(\nu)=\notag \\
&\sum_{i=1}^{I}
 (\phi(\langle\nu,fg_1g_2\rangle+f(x^i)g_1(n^i_1)g_2(n_2^i))
 -\phi(\langle\nu,fg_1g_2\rangle))B(x^i,n_1^i,n_2^i)(1-p(x^i)) \notag \\
 &+\sum_{i=1}^{I}\int
 (\phi(\langle\nu,fg_1g_2\rangle+f(x^i+z)g_1(n^i_1)g_2(n_2^i))
 -\phi(\langle\nu,fg_1g_2\rangle))  B(x^i,n_1^i,n_2^i)\,p(x^i)M(x^i, z)dz \notag \\
&+\sum_{i=1}^{I}
 (\phi(\langle\nu,fg_1g_2\rangle-f(x^i)g_1(n^i_1)g_2(n_2^i))
 -\phi(\langle\nu,fg_1g_2\rangle))(D(x^i,n_1^i,n_2^i)+\alpha(x^i,n^i_1,n^i_2)U*\nu(x^i,n_1^i,n_2^i)) \notag \\
&+\sum_{i=1}^{I}
 (\phi(\langle\nu,fg_1g_2\rangle+f(x^i)(g_1(n^i_1+1)-g_1(n^i_1))g_2(n_2^i))
 -\phi(\langle\nu,fg_1g_2\rangle))b_1(x^i)n^i_1 \notag \\
&+\sum_{i=1}^{I}
 (\phi(\langle\nu,fg_1g_2\rangle+f(x^i)g_1(n^i_1)(g_2(n_2^i+1)-g_2(n_2^i)))
 -\phi(\langle\nu,fg_1g_2\rangle))b_2(x^i)n^i_2\notag \\
 &+\sum_{i=1}^{I}
 (\phi(\langle\nu,fg_1g_2\rangle+f(x^i)(g_1(n^i_1-1)-g_1(n^i_1))g_2(n_2^i))
 -\phi(\langle\nu,fg_1g_2\rangle))\notag \\&\hskip 5cm(d_1(x^i)+\beta_1(x^i) ( \lambda_{11}n^i_1+ \lambda_{12}n^i_2))n^i_1\notag \\
&+\sum_{i=1}^{I}
 (\phi(\langle\nu,fg_1g_2\rangle+f(x^i)g_1(n^i_1)(g_2(n_2^i-1)-g_2(n_2^i)))
 -\phi(\langle\nu,fg_1g_2\rangle))\notag \\&\hskip 5cm(d_2(x^i)+\beta_2(x^i)( \lambda_{21}n^i_1+ 
 \lambda_{22}n^i_2))n^i_2.
\end{align}
The three first terms of~(\ref{generator}) capture the effects of births and deaths of individuals of
the population and  the for last
terms that of the cells. The competition makes the death terms
nonlinear.

\subsection{Process construction}
\label{sec:constr}

Let us give a pathwise construction of a Markov process admitting $L$ as
infinitesimal generator.

\me {\bf Assumptions (H1):}

\me {\it There exist constants $\bar{B}$, $\bar{D}$, $\bar{G}$
$\bar{\alpha}$, $\bar{U}$ and $\bar{C}$ and a probability density
function $\bar{M}$ on $\rit^d$  such that for $x,z\in {\cal X}$, $n_1, n_2\in
\mathbb{R}_{+}$, 
\begin{align*}
 &B(x,n_1,n_2)  \leq \bar{B}\ ;\notag\\ 
 &D(x,n_1,n_2)\leq \bar{D}\,(n_1+n_2)= \bar{D}\, n \ ;\notag\\ 
 & \alpha(x,n_1,n_2) \leq \bar{\alpha}\,(n_1+n_2)= \bar{\alpha}\, n\ ;\notag\\ 
& U(x)\leq \bar{U}, \  M(x,z)\leq \bar{C}\bar{M}(z).
\end{align*}
}
\noindent Remark that the jump rate of an individual
with  $n$ cells in the population $\nu$  is then upper-bounded by a constant times $n\
(1+ \langle\nu,1\rangle)$ and that the cell jump rate of such 
individual is upper-bounded by a constant times $ n (1+ n)$. Thus the model presents  a double nonlinearity since the population jump rates may depend on the product of  the size of the population times the number of cells and  quadratically on the number of cells. 

\bigskip \noindent 
Let us now give a pathwise
description of the population process $(\nu_t)_{t\geq 0}$. 

\begin{nota}
 \label{defh}
We associate to any population state $\nu=\sum_{i=1}^I \delta_{(x^i, n_1^i, n_2^i)}\in {\cal M}$
the triplet 
 $H^i(\nu)=(X^i(\nu), N_1^i(\nu), N_2^i(\nu))$
 as the trait and state of the $i$th-individual, obtained by ordering all triplets 
 with respect to some arbitrary order  on $\mathbb{R}^d\times \nit\times \nit$ (
 for example the lexicographic order).
\end{nota}

\noindent  We now introduce the probabilistic objects we will need.

\begin{defi}
 \label{poisson}
 Let $(\Omega, {\cal F}, P)$ be a probability
 space on which we consider
 the following  independent random elements:
 \begin{description}
 \item[\textmd{(i)}] a ${\cal M}$-valued random variable $\nu_0$ (the
   initial distribution),
 \item[\textmd{(ii)}] A Poisson point measure
   $Q(ds,di,dz,d\theta)$ on $\rit_+ \times \mathbb{N}^* \times
   {\cal X} \times \rit_+$ with intensity measure $\: ds
   \left(\sum_{k\geq 1} \delta_k (di) \right)\bar{M}(z) dz d\theta\:$.
 \end{description}
 Let us denote by $( {\cal F}_t)_{t\geq 0}$ the canonical filtration
 generated by $\nu_0$ and $Q$.\\

 \me Let us finally define the quantities $\theta^i_1(s)$, $\theta^i_2(s)$,
 $\theta^i_3(s)$, $\theta^i_4(s)$, $\theta^i_5(s)$, $\theta^i_6(s)$,
 $\theta^i_7(s)$ related to  the different  jump rates at time $s$ as:
 \begin{eqnarray}
 \label{theta}
\theta^i_1(s)&=& B(H^i(\nu_\sm))
(1-p(X^i(\nu_\sm)));\nonumber\\
\theta^i_2(s)-\theta^i_1(s)&=&B(H^i(\nu_\sm)) p(X^i(\nu_\sm))
\frac{M(X^i(\nu_\sm), z)}{\bar{M}(z)};\nonumber\\
\theta^i_{3}(s)-\theta^i_2(s)&=& D(H^i(\nu_\sm)) +  \alpha(H^i(\nu_\sm))\
U*\nu_\sm(X^i(\nu_\sm));\nonumber\\
\theta^i_4(s)-\theta^i_3(s)&=&b_1(X^i(\nu_\sm))\ N_1^i(\nu_\sm);\nonumber\\ \theta^i_5(s)-\theta^i_4(s)&=&b_2(X^i(\nu_\sm))\ N_2^i(\nu_\sm);\nonumber\\
\theta^i_6(s)-\theta^i_5(s)&=&d_1(X^i(\nu_\sm))+\beta_1(X^i(\nu_\sm))
(N_1^i(\nu_\sm)\lambda_{11} + N_2^i(\nu_\sm)\lambda_{12})\ N_1^i(\nu_\sm);\nonumber\\
\theta^i_7(s)-\theta^i_6(s)&=&d_2(X^i(\nu_\sm))+\beta_2(X^i(\nu_\sm))
(N_1^i(\nu_\sm)\lambda_{21} + N_2^i(\nu_\sm)\lambda_{22})\ N_2^i(\nu_\sm).\nonumber\\
\
\end{eqnarray}
\end{defi}

\noindent We finally define the population process in terms of these stochastic
objects.

\begin{defi}
 \label{dbpe}
 Assume $(H1)$. A $( {\cal F}_t)_{t\geq 0}$-adapted stochastic process
 $\nu=(\nu_t)_{t\geq 0}$  is called  a population process if a.s., for all $t\geq 0$,

%\newpage
 \begin{align}
   \nu_t &= \nu_0 + \int_{(0,t] \times \mathbb{N}^*
     \times
     {\cal X}\times\rit_+}\bigg\{\delta_{(X^i(\nu_\sm),N_1^i(\nu_\sm),N_2^i(\nu_\sm))}
     \indiq_{\{i \leq \left<
       \nu_\sm,1 \right>\}}  \indiq_{\left\{\theta \leq
      \theta^i_1(s)\right\}}
   \notag \\ &+ \delta_{(X^i(\nu_\sm)+z ,N_1^i(\nu_\sm),N_2^i(\nu_\sm))}
   \indiq_{\{i \leq \left<
       \nu_\sm,1 \right>\}} \indiq_{\left\{\theta^i_1(s)\leq \theta \leq
       \theta^i_2(s)
     \right\}} \notag \\ &- \delta_{(X^i(\nu_\sm),N_1^i(\nu_\sm),N_2^i(\nu_\sm))}
      \indiq_{\{i \leq
     \left< \nu_\sm,1 \right>\}} \indiq_{\left\{\theta^i_2(s)\leq \theta
     \leq\theta^i_3(s)
     \right\}}
      \notag \\ &+\bigg(\delta_{(X^i(\nu_\sm),N_1^i(\nu_\sm)+1,N_2^i(\nu_\sm))}
      -\delta_{(X^i(\nu_\sm),N_1^i(\nu_\sm),N_2^i(\nu_\sm))}\bigg)\indiq_{\{i \leq
     \left< \nu_\sm,1 \right>\}} \indiq_{\left\{\theta^i_3(s)\leq\theta \leq \theta^i_4(s)\right\}}
     \notag \\ &+\bigg(\delta_{(X^i(\nu_\sm),N_1^i(\nu_\sm),N_2^i(\nu_\sm)+1)}
      -\delta_{(X^i(\nu_\sm),N_1^i(\nu_\sm),N_2^i(\nu_\sm))}\bigg)\indiq_{\{i \leq
     \left< \nu_\sm,1 \right>\}} \indiq_{\left\{\theta^i_4(s)\leq\theta \leq \theta^i_5(s)\right\}}
     \notag \\ &+\bigg(\delta_{(X^i(\nu_\sm),N_1^i(\nu_\sm)-1,N_2^i(\nu_\sm))}
      -\delta_{(X^i(\nu_\sm),N_1^i(\nu_\sm),N_2^i(\nu_\sm))}\bigg)\indiq_{\{i \leq
     \left< \nu_\sm,1 \right>\}}
      \indiq_{\left\{\theta^i_5(s)\leq\theta \leq \theta^i_6(s)
     \right\}}
     \notag \\ &+\bigg(\delta_{(X^i(\nu_\sm),N_1^i(\nu_\sm),N_2^i(\nu_\sm)-1)}
      -\delta_{(X^i(\nu_\sm),N_1^i(\nu_\sm),N_2^i(\nu_\sm))}\bigg)\indiq_{\{i \leq
     \left< \nu_\sm,1 \right>\}} \indiq_{\left\{\theta^i_6(s)\leq \theta
     \leq\theta^i_7(s)
    \right\}}\bigg\}\notag \\
      &\hskip 12cm Q(ds,di,dz,d\theta)\notag \\
      \
   \label{bpe}
 \end{align}
\end{defi}

\noindent Let us now show that if $\nu$ solves~(\ref{bpe}), then $\nu$
follows the Markovian dynamics we are interested in.

\begin{prop}
 \label{gi}
 Assume $(H1)$ and consider a process $(\nu_t)_{t\geq 0}$ defined by 
  (\ref{bpe}) such that for all $T>0$, $\mathbb{E} (\sup_{t\leq
   T}\langle\nu_t, 1\rangle^3)<+\infty $ 
and $\mathbb{E} (\sup_{t\leq T}\langle\nu_t,n^2\rangle)<+\infty$.
 Then  $(\nu_t)_{t\leq 0}$ is a Markov process. Its infinitesimal generator
 $L$ applied to any bounded and measurable maps $\phi_ {fg_1g_2}: {\cal
   M}\mapsto \mathbb{R}$ and  $\nu \in {\cal M}$  satisfies (\ref{generator}). In particular, the law of $(\nu_t)_{t\geq 0}$
 does not depend on the chosen order in Notation \ref{defh}.
\end{prop}

\begin{proof}
 The fact that $(\nu_t)_{t\geq 0}$ is a Markov process is immediate.
 Let us now consider a function $\phi_ {fg_1g_2}$ as in the statement.
   Using  the decomposition (\ref{bpe}) of the measure $\nu_t$ and the fact that 
\be
\label{ito}
\phi_{fg_1g_2}(\nu_t) = \phi_{fg_1g_2}(\nu_0) + \sum_{s\leq t}
 (\phi_{fg_1g_2}(\nu_\sm + (\nu_s-\nu_\sm))-\phi_ {fg_1g_2}(\nu_\sm)) \quad \textrm{a.s.} ,
\ee
we get a decomposition of $\phi_{fg_1g_2}(\nu_t)$.  

 \noindent Thanks to the moment assumptions,
 $\phi_{fg_1g_2}(\nu_t)$ is integrable. Let us check it for the nonlinear individual death term (which is the more delicate to deal with):
\begin{align*} &\mathbb{E}\bigg(  \int_{(0,t] \times \mathbb{N}^*
     \times
     {\cal X}\times\rit_+} \big( \phi(\langle \nu_\sm- \delta_{(X^i(\nu_\sm),N_1^i(\nu_\sm),N_2^i(\nu_\sm))}, fg_{1} g_{2}\rangle - \phi(\langle \nu_\sm, fg_{1} g_{2}\rangle\big)
      \indiq_{\{i \leq
     \left< \nu_\sm,1 \right>\}} \\
     &\hskip 7cm \indiq_{\left\{\theta^i_2(s)\leq \theta
     \leq\theta^i_3(s)
     \right\}}  Q(ds,di,dz,d\theta)\bigg) \\
     &=\\
     & \mathbb{E}\bigg(  \int_{0}^t \langle \nu_{s}, \big( \phi(\langle \nu_s, fg_{1} g_{2}\rangle - f(x)g_{1}(n_{1}) g_{2}(n_{2})) - \phi(\langle \nu_s, fg_{1} g_{2}\rangle\big)
       (D(x,n_{1},n_{2}) +  \alpha(x,n_{1},n_{2})\
U*\nu_s(x))\rangle ds\bigg).
\end{align*}
Since $\phi$ is bounded and thanks to Assumption $(H1)$, the right hand side term will be finite as soon as
$$\mathbb{E}\left(\sup_{t\leq T} (\langle\nu_{t},n\rangle  +\langle\nu_{t},n\rangle \langle\nu_{t},1\rangle)\right)<\infty.$$
Remark firstly that  $\langle\nu,n\rangle \leq \langle\nu,n^2\rangle$. Moreover we get $n\langle\nu,1\rangle\leq 1/2(n^2+ \langle\nu,1\rangle^2)$ and thus 
\ben \langle\nu,n\rangle\, \langle\nu,1\rangle \leq 1/2(\langle \nu,n^2+ \langle\nu,1\rangle^2\rangle)= 1/2(\langle \nu,n^2\rangle + \langle\nu,1\rangle^3).
\een
 The moment assumptions allow us to conclude and to show that
 the expectation is 
 differentiable in time at $t=0$. It leads to~(\ref{generator}).
\end{proof}

\medskip
\noindent Let us show existence and moment properties for the population
process.

\begin{thm}\label{existence} Assume~(H1).

\textmd{(i)} If $\ \mathbb{E} \left( \left<\nu_0, 1 \right> \right)
<+\infty$, then the process $(\nu_t)_{t}$ introduced in Definition \ref{dbpe}
   is well defined on $\rit_+$.

\textmd{(ii)} Furthermore, if for some $p\geq 1$, $\mathbb{E} \left(
     \left<\nu_0,1 \right>^p \right) <+\infty$, then for any
   $T<\infty$,
   \be
     \label{lp}
     \mathbb{E} (\sup_{t\in[0,T]} \left< \nu_t,1\right>^p ) <+\infty.
    \ee

 \textmd{(iii)}  If moreover  
$\mathbb{E}  \left( \left<\nu_0,n^2 \right> \right) <+\infty$, then for any
   $T<\infty$,
   \be
   \label{n2}
\ \mathbb{E}  (\sup_{t\in[0,T]} \left< \nu_t,n^2\right> ) <+\infty.
   \ee
   \end{thm}

\begin{proof}
We compute $ \phi(<\nu_t,1>)$ using \eqref{bpe} and \eqref{ito} for $f\equiv g_1\equiv g_2 \equiv 1$: we get
\begin{align}
   \phi(<\nu_t,1>) & = \phi(<\nu_0,1>)+\int_{(0,t] \times \mathbb{N}^*
     \times
     {\cal X}\times\rit_+}\bigg\{\left(\phi(<\nu_\sm,1>+1)
     - \phi(<\nu_\sm,1>) \right) \indiq_{\left\{\theta \leq
      \theta^i_2(s)\right\}}\notag\\&
     + \left(\phi(<\nu_\sm,1>-1)
     - \phi(<\nu_\sm,1>) \right)\indiq_{\left\{\theta^i_2(s) \leq \theta \leq
      \theta^i_3(s)\right\}}\bigg\} \indiq_{\{i \leq
     \left< \nu_\sm,1 \right>\}} Q(ds,di,dz,d\theta)
     \label{phinu}
   \end{align}
and for $g_1(n_1)=n_1$,
\begin{align}
   <\nu_t,n_1> & = <\nu_0,n_1>+\int_{(0,t] \times \mathbb{N}^*
     \times
     {\cal X}\times\rit_+} \bigg\{N^i_1(\nu_\sm)\bigg(\indiq_{\left\{\theta \leq
      \theta^i_2(s)\right\}}-
     \indiq_{\left\{\theta^i_2(s) \leq \theta \leq
      \theta^i_3(s)\right\}}\bigg)\notag\\ &+ \indiq_{\left\{\theta^i_3(s) \leq \theta \leq
      \theta^i_4(s)\right\}}-\indiq_{\left\{\theta^i_5(s) \leq \theta \leq
      \theta^i_6(s)\right\}}\bigg\} \indiq_{\{i \leq
     \left< \nu_\sm,1 \right>\}} Q(ds,di,dz,d\theta).
     \label{nun}
   \end{align}
A similar decomposition holds for $ <\nu_t,n_2> $.\\
The proof of ({\it i}) and  ({\it ii}) is standard and can easily be adapted from \cite{FM04}:
% Consider the process $(\nu_t)_{t\geq 0}$.
 we
 introduce for each integer $k$ the stopping time $\tau_k = \inf \left\{ t
   \geq 0 , \; \left<\nu_t,1 \right> \geq k \right\}$ and show that the sequence$(\tau_{k})_{k}$ tends a.s. to infinity, using that 
      \begin{align*}
  & \sup_{s\in[0,t\land \tau_k]}  \left< \nu_s,1\right>\leq \left< \nu_0,1\right> + \int_{(0,t\land \tau_k] \times
     \mathbb{N}^* \times {\cal X}\times \rit^+}  \indiq_{\{i \leq \left< \nu_\sm,1
     \right>\}}
  \ \indiq_{\left\{\theta \leq
      \theta_2^i(s)\right\}}
    Q(ds,di,dz,d\theta),
 \end{align*}
   and the estimates of moments up to time $\tau_{k}$ deduced from the latter and Assumption $(H1)$ and   Gronwall's lemma.  
  
\noindent Further, one may build the
 solution $(\nu_t)_{t\geq 0}$ step by step. One only has to check
 that the sequence of jump instants $(T_k)$ goes a.s.\ to infinity as
 $k$ tends to infinity, which follows from the previous
 result.\\

\noindent  The proof of ({\it iii}) follows a similar argument  with   $\tau_k^1 := \inf \left\{ t
   \geq 0 , \; \left<\nu_t,n_1^2 \right> \geq k \right\}$.  From
 \begin{align*}
  & \sup_{s\in[0,t\land \tau_k^1]}  \left< \nu_s,n_1^2\right>\leq \left<
     \nu_0,n_1^2\right> + \int_{(0,t\land \tau_k^1] \times
     \mathbb{N}^* \times {\cal X}\times \rit^+}  \indiq_{\{i \leq \left< \nu_\sm,1
     \right>\}}
     \\
  &\hskip 2cm \bigg\{( N_1^i(\nu_\sm))^2\,\indiq_{\left\{\theta \leq
      \theta_2^i(s)\right\}} + (2\,  N_1^i(\nu_\sm) + 1)\,\indiq_{\left\{\theta^i_3(s)\theta \leq
      \theta^i_4(s)\right\}}\bigg\}
    Q(ds,di,dz,d\theta),
 \end{align*}
and  $\mathbb{E} \left(
     \left<\nu_0,n^2 \right> \right) <+\infty$ and ({\it ii}) since $2 n_{1} + 1 \leq n_{1}^2 +2$, we firstly get,  using Assumption $(H1)$ and  Gronwall's lemma, that
 $$ \mathbb{E} ( \sup_{t \in [0,T\land \tau_k^1]} \left< \nu_t,n_{1}^2\right>)\leq C_{T}.$$
 Then we deduce that $\tau_k^1$ tends to infinity a.s. and that
     $\mathbb{E} (\sup_{t\in[0,T]} \left< \nu_t,n_1^2\right> ) <+\infty$.
     The same is true replacing $n_1$ by $n_2$.
\end{proof}

%%%%%%%%%%%%%%%%%%%%%%%%%%%%%%%%%%%%%%%%%%%%%%%%%%%%
\subsection{Martingale Properties}
\label{secexist}
%%%%%%%%%%%%%%%%%%%%%%%%%%%%%%%%%%%%%%%%%%%%%%%%%%%%%

We finally give some martingale properties of the process
$(\nu_t)_{t\geq 0}$, which  are the key point of our approach. For measurable functions
$f, g_1, g_2$, let us denote by $F_{fg}$ the function defined on
$M_F$ by 
$$
F_{fg}(\nu):= <\nu,fg_1g_2>
%(=\phi_{fg_1g_2} (\nu) \textrm{ with } \phi={\rm Id})
.
$$
\begin{thm}
 \label{martingales}
 Assume $(H1)$ together with  $\mathbb{E}  \left( \left<
     \nu_0,1 \right>^3 \right) <+\infty$ and $\mathbb{E} \left( \left<
     \nu_0,n^2\right> \right) <+\infty$.\\
\noindent (i)  For all measurable  functions $\ \phi, f, g_1, g_2\ $ such
that\\
$\ | \phi_{fg_1g_2} (\nu) | +|L \phi_{fg_1g_2}(\nu)|\leq
C(1+ \langle \nu,1\rangle^3 +\langle \nu,n^2\rangle)$, the process
   \begin{equation}
     \label{pbm1}
   \phi_{fg_1g_2} (\nu_t) - \phi_{fg_1g_2} (\nu_0) - \intot   L \phi_{fg_1g_2}(\nu_s)
     ds
   \end{equation}
   is a c\`adl\`ag  $({\cal F}_t)_{t\geq 0}$-martingale starting from $0$,
 where $L \phi_{fg_1g_2}$ has been defined in
 (\ref{generator}).\\

\noindent  (ii) For all measurable bounded functions $f, g_1, g_2$, the process
   \begin{equation}
     \label{pbm2}
    M_t^{fg}= \left<\nu_t,fg_1g_2\right> - \left<\nu_0,fg_1g_2\right> - \intot  LF_{fg}(\nu_s)
     ds
   \end{equation}
   is a c\`adl\`ag square integrable $({\cal F}_t)_{t\geq 0}$-martingale starting from $0$,
 where
   \begin{align}
    LF_{fg}(\nu) &=
     \int_{{\cal X} \times \nit^2}
     \bigg\{\bigg(B(x,n_1,n_2)(1-p(x,n_1,n_2))
     -(D(x,n_1,n_2)+\alpha(x,n_1,n_2)U*\nu (x))\bigg)\notag \\
     &\hskip 5cm f(x)
     g_1(n_1)g_2(n_2)
      \notag \\
     &+p(x,n_1,n_2)B(x,n_1,n_2)\int_{}f(x+z)g_1(n_1)g_2(n_2)M(x,n_1,n_2,z)dz  \notag \\
     &+ f(x)\big(g_1(n_1+1)-g_1(n_1)\big)g_2(n_2)b_1(x)n_1+
      f(x)g_1(n_1)\big(g_2(n_2+1)-g_2(n_2)\big)b_2(x)n_2\notag \\
      &+ f(x)\big(g_1(n_1-1)-g_1(n_1)\big)g_2(n_2)\big(d_1(x)+\beta_1(x)(\lambda_{11}n_1 + \lambda_{12}n_2)\big)n_1\notag \\
     & +
     f(x)g_1(n_1)\big(g_2(n_1-1)-g_2(n_2)\big)\big(d_2(x)+\beta_2(x)(\lambda_{21}n_1 +\lambda_{22}n_2)\big)n_2\bigg\}\nu(dx,dn_1,dn_2).
     \label{gen}
   \end{align}

\noindent   Its
    quadratic variation is given by 
   \begin{align}
     \langle M^{fg}\rangle_t &= \intot \int_{{\cal  X}\times \nit^2}
     \bigg\{\bigg((1-p(x,n_1,n_2))B(x,n_1,n_2)+(D(x,n_1,n_2)+\alpha(x,n_1,n_2)U*\nu_s(x))\bigg)\notag \\
     & \hskip 5cm
   f^2(x)g_1^2(n_1)g_2^2(n_2)
     \notag \\
     &+p(x,n_1,n_2)B(x,n_1,n_2)\int_{}f^2(x+z)g_1^2(n_1)g_2^2(n_2)M(x,n_1,n_2,z)dz
     \notag \\
    & + f^2(x)\big(g_1(n_1+1)-g_1(n_1)\big)^2g_2^2(n_2)b_1(x)n_1\notag \\
     &+
      f^2(x)g_1^2(n_1)\big(g_2(n_2+1)-g_2(n_2)\big)^2b_2(x)n_2\notag \\
      &+ f^2(x)\big(g_1(n_1-1)-g_1(n_1)\big)^2g_2^2(n_2)\big(d_1(x)+\beta_1(x)(\lambda_{11}n_1 +\lambda_{12}n_2)\big)n_1\notag \\
     & + f^2(x)g_1^2(n_1)\big(g_2(n_1-1)-g_2(n_2)\big)^2\big(d_2(x)+\beta_2(x)(\lambda_{21}n_1 + \lambda_{22}n_2)\big)n_2
     \bigg\} \nu_s(dx,dn_1,dn_2) ds. \label{qv}
   \end{align}
\end{thm}

\begin{proof}
The martingale property is immediate  by
 Proposition~\ref{gi} and Theorem \ref{existence}. Let us justify the form of the quadratic variation process. 
Using  a localization argument as in Theorem \ref{existence}, we may
 compare two different expressions of $ \langle \nu_t, fg_1g_2\rangle^2$. The
 first one is obtained by  applying~(\ref{pbm1}) with $\phi(\nu) := \left< \nu,
 fg_1g_2\right>^2$. The second one is obtained
 by applying It\^o's formula to
 compute $\left< \nu_t, fg_1g_2\right>^2$ from~(\ref{pbm2}).  
Comparing these expressions leads  to~(\ref{qv}). 
We may let go the localization stopping time sequence to infinity since  $E \big(  \left< \nu_0,1 \right>^3 \big)<+\infty$ and   
$E\left( \left< \nu_0,n^2\right> \right) <+\infty$. Indeed,
 in this case, $E(\langle M^{fg} \rangle_t)<+\infty$ thanks
 to Theorem \ref{existence} and to the proof of Proposition \ref{gi}.
\end{proof}

%%%%%%%%%%%%%%%%%%%%%%%%%%%%%%%%%%%%%%%%%%%%%%%%%%%%%%%%%%%%%%%%%%%%%%%%%%%
\section{Deterministic large population approximations}
\label{sec:large-popu}

%%%%%%%%%%%%%%%%%%%%%%%%%%%%%%%%%%%%%%%%%%%%%%%%%%%%%%%%%%%%%%%%%%%%%%%%%%%

\noindent We are interested in studying  large population
approximations of our individual-based system.
We rescale the size of individual population by $K$ and the size of the cell populations by $ K_{1}$ respectively $ K_{2}$. 
 With $\kappa =( K, K_{1}, K_{2})$, the process of interest is now
the Markov process $(Y^\kappa_t)_{t\geq 0}$ defined as
\begin{equation*}
 Y^\kappa_t= \frac{1}{K} \sum_{i=1}^{I_\kappa(t)} \delta_{(X^i_\kappa(t),\frac{N_{1,\kappa}^i(t)}{K_1}, \frac{N_{2,\kappa}^i(t)}{K_2})} 
\in M_F({\cal X}\times\rit_+ \times \rit_+)
\end{equation*}
in which cells of type $1$ (resp. of type $2$) have been weighted by ${1\over K_{1}}$  (resp. by ${1\over K_{2}}$) and individuals by ${1\over K}$. 
The dynamics of the process $(X^i_\kappa(t), N_{1,\kappa}^i(t), N_{2,\kappa}^i(t))$ is the one described in Section $2$ except some coefficients are depending on the scaling $\kappa$ as described below.\\
The individual dynamics depends on $B_\kappa$, $p_\kappa$, $M_\kappa$, $D_\kappa$, $\alpha_\kappa$, $U_\kappa$ which are assumed to satisfy the assumptions $(H_{1})$ of the section 2 for any fixed $\kappa$.

\medskip \noindent {\bf Notation}: We say that $\kappa \to \infty$ when the three parameters $K, K_1, K_2$ tend to infinity.

\me \noindent {\bf Assumptions (H2):}  \\
{\it 1) There exist continuous functions $B$,  $D$ and $\alpha$ on ${\cal X} \times \rit_+\times \rit_+$ such that
\begin{align}
\label{scaling}
&\lim_{\kappa \to \infty}  \sup_{x,y_{1}, y_{2}}|B_\kappa(x, K_1 y_1,K_2 y_2)-B(x,y_1,y_2)|+ 
|D_\kappa(x,  K_1 y_1,K_2 y_2)-D(x,y_1,y_2)|=0,\nonumber\\
&\lim_{\kappa \to \infty}  \sup_{x,y_{1}, y_{2}}|\alpha_\kappa(x, K_1 y_1,K_2 y_2)-\alpha(x,y_1,y_2) |=0.
\end{align}

\noindent
We assume that the functions $B$, $D$ and $\alpha$ satisfy Assumption $(H1)$. 

\medskip \noindent 2) The competition kernel $U_{\kappa}$ satisfies
\be
\label{resource}
U_\kappa(x)={U(x)\over K},\ee
where $U$ is a continuous function.  

\me \noindent 3) The others parameters $p_{\kappa}=p$ and $M_{\kappa}=M$ stay unchanged, as also the cell ecological parameters:
$b_{1,\kappa}=b_{1}$,
$b_{2,\kappa}=b_{2}$, $d_{1,\kappa}=d_{1}$,$d_{2,\kappa}=d_{2}$, $\beta_{1,\kappa}=\beta_{1}$, $\beta_{2,\kappa}=\beta_{2}$. 
The functions $p$ and $M$ are assumed to be continuous and the  functions $b_{i}, d_{i}$ and $\beta_{i}$ are of class $C^1$. 
\\
We assume
$$r_{i}= b_{i} - d_{i}>0\ , \quad i \in \{1,2 \}.$$

\noindent 4) Similarly to \eqref{resource}, the interaction rates between cells satisfy
\be
\label{inter-cell}
\lambda_{ij}^\kappa=\frac{\lambda_{ij} }{K_j}, \qquad i,j \in \{1,2\} .
\ee
}

Remark that Assumption (H2) 1) means that at a large scale $K$,  the individuals are influenced in their ecological behavior by the cells if the number of the latter is of order $K_{1}$  for cells of type $1$, resp. of order $K_{2}$  for cells of type $2$. On the other side the hypothesis (H2) 2) may be a consequence of a fixed amount
of available resources to be partitioned among all the individuals. Larger
systems are made up of smaller interacting individuals whose
biomass is scaled by $1/K$, which  implies that the interaction
effect of the global population on a focal individual is of order
$1$. 

\bi \noindent {\bf Examples} 

\me (i) If $K_{1} = K_{2} $ and if the individual rates $B_\kappa, D_\kappa, \alpha_\kappa$ only depend on $x, n_{1}, n_{2}$ by the proportion of cells of type $1$, then \eqref{scaling} is satisfied.

\me (ii) Assume that  $K_{1} = K_{2} = K$ and that the functions $B_\kappa, D_\kappa, \alpha_\kappa$  only depend on the weighted total number of cells  ${1\over K} (n_{1}+n_{2})$.

\subsection{A convergence theorem}

\noindent We assume that  the sequence of random initial conditions
$Y_0^\kappa$ converges in law  to some finite measure
$v_0 \in M_F({\cal X}\times\rit_+ \times \rit_+)$  when  $\kappa \to \infty$.  Our aim is  to study the limiting behavior of the processes $Y^\kappa_\cdot$ as $\kappa \to \infty$.

\noindent The generator $L^\kappa$ of
$(Y^\kappa_t)_{t\geq 0}$ is easily obtained by computing, for any measurable
function $\phi$ from $M_F({\cal X}\times\rit_+ \times \rit_+)$ into $\rit$ and any $\mu
\in M_F({\cal X}\times\rit_+ \times \rit_+)$,
\begin{equation*}
 L^\kappa\phi(\mu)=\partial_t \mathbb{E}_{\mu}(\phi(Y^\kappa_t))_{t=0}.
\end{equation*}

\noindent In particular, similarly  as in Theorem~\ref{martingales}, we may
summarize the moment and martingale properties of $Y^\kappa$.

\begin{prop}
\label{YK}
Assume that for some $p\geq 3$,  $\mathbb{E}(\langle Y^\kappa_0,1\rangle^p + \langle Y^\kappa_0,y_1^2+y_2^2\rangle)<+\infty$. Then
\begin{description}
\item[\textmd{(1)}] For any $T>0$, $\mathbb{E}\left(\sup_{t\in[0,T]}\langle Y^\kappa_t,1\rangle^p +\sup_{t\in[0,T]}\langle Y^\kappa_t,y_1^2+y_2^2\rangle\right)<+\infty$.
\item[\textmd{(2)}]
For any measurable bounded functions $f, g_1, g_2$, the process
\begin{align}  \label{pbmYK}
&\tilde M_t^{\kappa,fg}  = \left<Y^\kappa_t,fg_1g_2\right> - \left<Y^\kappa_0,fg_1g_2\right>  
- \int_0^t \int_{{\cal  X}\times \rit^2_+}
   \bigg\{\Big( B_\kappa(x,K_1 y_1,K_2 y_2)(1-p(x)) \notag \\
&  \qquad -\big(D_\kappa(x,K_1 y_1,K_2 y_2)+ \alpha_\kappa(x,K_1 y_1,K_2 y_2)\ U*Y^\kappa_s(x,y_1,y_2)\big)\Big)
f(x)g_1(y_1)g_2(y_2)  \notag \\
& \qquad + p(x)B_\kappa(x,K_1 y_1,K_2 y_2)\int_{}f(x+z)g_1(y_1)g_2(y_2)M(x,z)dz  \notag \\
& \qquad + f(x)\big(g_1(y_1+\frac{1}{K_1})-g_1(y_1)\big)g_2(y_2)\ b_{1}(x)\  K_1 y_1 \notag \\
&\qquad + f(x)g_1(y_1)\big(g_2(y_2+\frac{1}{K_2})-g_2(y_2)\big)\ b_{2}(x)\ K_2 y_2\notag \\
& \qquad + f(x)\big(g_1(y_1-\frac{1}{K_1})-g_1(y_1)\big)g_2(y_2)
 \Big(d_{1}(x)+\beta_{1}(x)(\lambda_{11}y_1+\lambda_{12}y_2)\Big) K_1 y_1\notag \\
& \qquad + f(x)g_1(y_1)\big(g_2(y_1-\frac{1}{K_2})-g_2(y_2)\big)
  \Big(d_{2}(x)+\beta_{2}(x)(\lambda_{21}y_1+\lambda_{22}y_2)\Big) K_2 y_2\bigg\} \notag \\
& \hskip 8.3cm Y_s^\kappa(dx,dy_1,dy_2)\ ds 
\end{align}
is a  c\`adl\`ag square integrable martingale starting from $0$
with quadratic variation
\begin{align}
& \langle \tilde M^{\kappa,fg}\rangle_t = 
\frac{1}{K}\intot \int_{{\cal  X}\times \rit^2_+}
\bigg\{\Big( B_\kappa(x,K_1 y_1,K_2 y_2)(1-p(x)) \notag \\
& \qquad  \qquad + \big(D_\kappa(x,K_1 y_1,K_2 y_2)+ \alpha_\kappa(x,K_1 y_1,K_2 y_2)\ U*Y^\kappa_s(x,y_1,y_2)\big)\Big)
f^2(x)g_1^2(y_1)g_2^2(y_2) \notag \\
& + p(x)B_\kappa(x,K_1 y_1,K_2 y_2)\int_{}f^2(x+z)g_1^2(y_1)g_2^2(y_2)M(x,z)dz  \notag \\
& + f^2(x)\ \big( g_1(y_1+\frac{1}{K_1})-g_1(y_1) \big)^2\ g_2^2(y_2)\ b_{1}(x)\ K_1 y_1 \notag \\
& +     f^2(x)\ g_1^2(y_1)\ \big (g_2(y_2+\frac{1}{K_2})-g_2(y_2) \big)^2\ b_{2}(x)\ K_2 y_2 \notag \\
& + f^2(x)\ \big( g_1(y_1-\frac{1}{K_1})-g_1(y_1) \big)^2\ g_2^2(y_2)\
\Big(d_{1}(x)+\beta_{1}(x)(\lambda_{11}y_1 +\lambda_{12}y_2 )\Big) K_1 y_1 \notag \\
& + f^2(x)\ g_1^2(y_1)\ \big( g_2(y_2-\frac{1}{K_2})-g_2(y_2) \big)^2\
 \Big(d_{2}(x)+\beta_{2}(x)(\lambda_{21}y_1+ \lambda_{22}y_2 )\Big) K_2 y_2
 \bigg\} \notag \\
& \hskip 9.3cm Y^\kappa_s(dx,dy_1,dy_2)\ ds.  \label{qv2}
\end{align}
\end{description}
\end{prop}

\noindent We can now state our convergence result.

\begin{thm}
\label{largepoplargecells}
Assume (H2). 
Assume moreover that the sequence of initial
conditions $Y^\kappa_0 \in M_F({\cal  X}\times \rit_+^2)$ satisfies
$ \sup_\kappa  \mathbb{E}(\langle Y^\kappa_0,1\rangle^3)<+\infty$ and
$\sup_\kappa \mathbb{E}(\langle Y^\kappa_0,y_1^2+y_2^2 \rangle)<+\infty$.
If  $Y^\kappa_0$ converges in law,  
as $\kappa$ tends to infinity, to a finite deterministic measure
$v_0$, then the sequence of processes $(Y^\kappa_t)_{0 \leq t\leq T}$ converges in
law in the Skorohod space $\dit([0,T],M_F({\cal  X}\times \rit_+^2))$, as $\kappa$ goes
to infinity, to the unique (deterministic) measure-valued flow  $v \in
C([0,T],M_F({\cal  X}\times \rit_+^2))$ satisfying  for any bounded and continuous function $f$ and any bounded functions $ g_1,g_2$ of class $C^1_{b}$, \begin{align}
\label{eq:limit}
& \langle v_t,fg_1g_2\rangle  =\langle v_0,fg_1g_2\rangle  +\int_0^t\int_{{\cal  X}\times \rit_+^2} \bigg\{\Big(B(x,y_1,y_2)(1-p(x)) \notag \\
& \hskip 3.8cm -\big(D(x,y_1,y_2)+\alpha(x,y_1,y_2)\ U*v_s(x,y_1,y_2)\big)\Big)
  f(x) g_1(y_1)g_2(y_2) \notag \\
& \qquad + p(x)B(x,y_1,y_2)\int_{}f(x+z)M(x,z)dz \ g_1(y_1)g_2(y_2) \notag \\
& \qquad + f(x) \Big[g'_1(y_1) g_2(y_2) b_{1}(x)\ y_1
 +  g_1(y_1)  g'_2(y_2) b_{2}(x)\ y_2  \notag \\
& \qquad -  g'_1(y_1)g_2(y_2)\Big(d_{1}(x)+\beta_{1}(x)(\lambda_{11}y_1+\lambda_{12}y_2)\Big)y_1\notag \\
& \qquad -g_1(y_1) g'_2(y_2)\Big(d_{2}(x)+\beta_{2}(x)(\lambda_{21}y_1+\lambda_{22}y_2)\Big)y_2\Big]\bigg\}
     v_s(dx,dy_1,dy_2)\  ds.
\end{align}
\end{thm}

\me Note that for this dynamics, a transport term appears at the
level of cells.\\

\begin{rema} 
\rm 
\begin{itemize}
\item 
A solution of \eqref{eq:limit} is a measure-valued solution of the nonlinear integro-differential equation 
   \begin{equation} \label{edptransport}
   {\partial \over \partial t} v_{t} = \Big(B(1-p) -\big(D+\alpha\ U*v_t\big)\Big)v_{t}  +  (B\,p\,  v_{t})* M 
  - \triangledown_y \cdot \big( c v_{t}\big)
   \end{equation}
   with
\begin{eqnarray}
c_1(x,y)&:=&  y_1 \left(r_1(x) - \beta_1(x) \big( \lambda_{11} y_1 + \lambda_{12} y_2 \big)\right)\nonumber \\ 
c_2(x,y)&:=&  y_2\left(r_2(x)  - \beta_2(x) \big( \lambda_{21} y_1 + \lambda_{22} y_2 \big) \right).\label{def-c}  
\end{eqnarray} 
Thus, the existence of a weak solution for Equation \eqref{edptransport} is obtained as corollary of Theorem \ref{largepoplargecells}.

\me 
   \item We deduce from (\ref{eq:limit}) the limiting dynamics of the total number of individuals:\begin{align}
\label{totalmass} 
& \langle v_t,1\rangle  =\langle
v_0,1\rangle  +\int_0^t\int_{{\cal  X}\times \rit_+^2}
\Big( B(x,y_1,y_2) - D(x,y_1,y_2) \notag\\
& \hskip 3.5cm -\alpha(x,y_1,y_2)\ U*v_s(x,y_1,y_2)\Big)
v_s(dx,dy_1,dy_2)\  ds,
\end{align}
while the total number $\langle v_t,y_i\rangle$ of cells of type $i$ at time $t$  is obtained by taking $f\equiv 1, g_i(y)=y, g_j\equiv 1$ ($i\neq j)$ in (\ref{eq:limit})~:
\begin{align}
\label{cellmassgen}
& \langle v_t,y_i\rangle  =\langle
v_0,y_i\rangle \notag \\
& + \int_0^t
\int_{{\cal  X}\times \rit_+^2} \Big( B(x,y_1,y_2) - D(x,y_1,y_2)-\alpha(x,y_1,y_2)\ U*v_s(x,y_1,y_2)\Big) 
y_i \ v_s(dx,dy_1,dy_2) ds   \notag \\
&  + \int_0^t\int_{{\cal  X}\times 
\rit_+^2}\Big( (b_{i}(x)-d_i(x))y_i
-\beta_{i}(x)(\lambda_{ii}y_i+\lambda_{ij}y_j)y_i \Big)
     v_s(dx,dy_1,dy_2)\  ds.
\end{align}
\end{itemize}
\end{rema}

\bigskip \noindent 

\begin{proof}
The proof of the theorem is obtained by a standard compactness-uniqueness result (see e.g. \cite{Ethier-Kurtz}).
The compactness is a consequence, using Prokhorov's Theorem, of the uniform tightness of the sequence of laws of $(Y^\kappa_{t}, t\geq 0)$.  This uniform tightness derives  
from uniform moment estimates. Their  proof is standard and we refer for details to \cite{Ro86}, \cite{FM04} Theorem 5.3  or to \cite{ChF06}. 
To identify the limit, we first remark using \eqref{qv2} that the quadratic variation tends to $0$ when $K$ tends to infinity. Thus the limiting values are deterministic and it remains to prove the convergence of the drift term in \eqref{pbmYK} to  the one in \eqref{eq:limit}. 
The drift term in  \eqref{pbmYK}  has the form $\int_0^t \langle Y^{\kappa}_{s}, A^\kappa(Y^\kappa_{s})(fg_{1}g_{2}) \rangle ds $ and the limiting term in \eqref{eq:limit} has the form $\int_0^t \langle v_{s}, A(v_{s})(fg_{1}g_{2})\rangle ds$. (The exact values of $A^\kappa$ and $A$ are immediately given by   \eqref{pbmYK} and \eqref{eq:limit}). 

\noindent Thus, let us show that if  $Y^\kappa$ is a sequence of random measure-valued processes weakly converging to a measure-valued flow $Y$ and satisfying the moment assumptions 
\be
\label{asmom}
\sup_{\kappa} \mathbb{E}(\sup_{t\leq T}\langle Y^{\kappa}_{t},1\rangle^3) \, + \, \sup_{\kappa} \mathbb{E}(\sup_{t\leq T}\langle Y^{\kappa}_{t},y^2\rangle) <+ \infty,\ee
 then  $\langle Y^{\kappa}_{t}, A^\kappa(Y^\kappa_{t})(fg_{1}g_{2})\rangle$  converges in $L^1$ to  $\langle Y_{t}, A(Y_{t})(fg_{1}g_{2})\rangle$  uniformly in time $t\in [0, T]$. We write
\be
\label{conv-gen}
&&\langle Y^{\kappa}_{t}, A^\kappa(Y^\kappa_{t})(fg_{1}g_{2})\rangle - \langle Y_{t}, A(Y_{t})(fg_{1}g_{2})\rangle\notag\\
&&= \langle Y^{\kappa}_{t}, A^\kappa(Y^\kappa_{t})(fg_{1}g_{2}) - A(Y^\kappa_{t})(fg_{1}g_{2})\rangle +\langle Y^{\kappa}_{t}, A(Y^\kappa_{t})(fg_{1}g_{2})-A(Y_{t})(fg_{1}g_{2})\rangle\notag\\
&& \hskip 0.5cm+ \langle Y^{\kappa}_{t} - Y_{t}, A(Y_{t})(fg_{1}g_{2})\rangle.
\ee
The convergence of the first term to zero follows from Assumptions $(H2)$ and \eqref{asmom} and from the following remark, that for $C^1_{b}$-functions $g_{1}$ and $g_{2}$, the terms 
$$K_{i} \Big( g_{i}(y_{i}-{1\over K_{i}} ) - g_{i}(y_{i})\Big) + g'_{i}(y_{i})$$
converge to $0$ in a bounded pointwise sense, which allows us  to apply the Lebesgue's theorem. 

\noindent The convergence of the second term to $0$ is immediately obtained by use of \eqref{asmom}, since the functions $\alpha$ and  $U$ are continuous and bounded. 

\noindent The convergence of the third term of  \eqref{conv-gen} is due to the weak convergence of $Y^\kappa$ to $Y$. We know that for all bounded and continuous functions $\phi$, the quantity $\langle Y^\kappa_{t} - Y_{t}, \phi\rangle$  tends
 to $0$. The function $A(Y_{t})(fg_{1}g_{2})$ is a continuous function which is not bounded because of  linear terms in $y$ and  $y^2$. Thus we need  to cutoff at a level $M$ replacing $y$ by $y\wedge M$.  The remaining terms  are proved to go to $0$ using  \eqref{asmom}. Hence  we have proved that each limiting value satisfies \eqref{eq:limit}.

\me \noindent We have now to prove the uniqueness of the solutions $v \in C([0,T],M_F({\cal  X}\times \mathbb{R}_+^2))$ of \eqref{eq:limit}. Our argument is based on properties of  Lotka-Volterra's flows. 
Firstly we need the following comparison lemma.

\begin{lem} \label{ineglogistique}
If $u_t$ is a non negative function with positive initial value and satisfying for some $a,b\in \mathbb{R}_+^*$ the inequality 
\begin{eqnarray*} 
\forall t>0, \quad \frac{\partial }{\partial t} u_t &\leq & a u_t  - b u_t^2, \\
\textrm{ then } \qquad  0 & \leq & \sup_{t \geq 0} u_t =: \overline u   < + \infty .
\end{eqnarray*} 

Moreover 0 is an absorbing value: if $u_{t_0}=0$ then  for all $ t \geq t_0, u_{t}\equiv 0$.
\end{lem}
\begin{proof}{\bf of  Lemma \ref{ineglogistique}}.
Let us define 
$U_t $ as solution of the associated logistic equation
$$
\frac{\partial  U_t}{\partial t} = a U_t - b U_t^2, \quad U_0= u_0 .
$$
Then
$
\frac{\partial }{\partial t}(U_t - u_t)
\geq    a(U_t -u_t) - b (U_t^2 - u_t^2).
$
 With $\delta_t:= U_t - u_t$ it holds
$$
\frac{\partial }{\partial t}\delta_t
\geq    \Big( a - b (U_t + u_t) \Big)\delta_t, \quad \delta_0=0.
$$
Let us show that $t\mapsto \delta_t$ increases, and therefore is positive.
For $t= 0$, since $\delta_0=0, \,\frac{\partial \delta_t}{\partial t}|_{t=0}  \geq   0 $. Thus $\delta_t \geq 0$ in a neighborhood of  0.\\
Let define  $t_0:= \sup\{t>0 :\delta_t=0 \}$. If $t_0 = + \infty$ the  problem is solved. \\
If not,  $  U_t \equiv  u_t $ on $[0, t_0]$.
Let us now define $t_1:= \inf\{t>t_0 :\delta_t<0 \}$. If $t_1 = + \infty$  the  problem is solved. If $t_1 < +\infty$,  by continuity $\delta_{t_1}=0$ and then $\frac{\partial \delta_t}{\partial t}|_{t=t_1} \geq 0$. Thus,
in a small time intervall  after $t_1$, $\delta_t$ would increase  and be positive, which is a contradiction with the definition of $t_1$.
Therefore  $t\mapsto \delta_t$ increases and stays positive, which implies that
$$
0 \leq \overline{u} := \sup_{t \geq 0} u_t \leq \sup_{t \geq 0} U_t < + \infty.
$$
\end{proof}\\
Let us now recall some properties of the Lotka-Volterra's flow involved in the cell dynamics.
\begin{lem}\label{lotka-volterra} 
Let $t_{0}\in [0,T]$, $x\in {\cal  X} $ and $y=( y_1, y_2)\in \mathbb{R}_+^2 $ be given.
The differential equation
\begin{eqnarray} \label{eq:LVdim2} 
\frac{\partial }{\partial t} y(t) = c(x, y(t)),\ t \in [t_0,T], \textrm{ with } 
y(t_0)= y
\end{eqnarray} where $c$ defined in \eqref{def-c},
 admits  in $\mathbb{R}_+^2 $ a unique solution   $t \mapsto\varphi^{t_{0},y}_{x}(t)=(\varphi^{t_{0},y}_{x,1}(t),\varphi^{t_{0},y}_{x,2}(t))$. Moreover the mapping $(x, t, s, y) \mapsto  \varphi^{s,y}_{x}(t)$ is $C^0$ in $x\in {\cal  X} $ and $C^\infty$ in $t,s,y\in [0,T]^2\times \mathbb{R}_+^2$ and is a characteristic flow in the sense that for all $s, t, u$, 
\be
\label{flow}
\varphi^{s,y}_{x}(t) = \varphi^{u,z}_{x}(t), \quad \hbox{ where } z = \varphi^{s,y}_{x}(u).
\ee
\end{lem}

\begin{proof}{\bf of  Lemma \ref{lotka-volterra}}.
Since the coefficients $c_{i}$ are of class $C^1$ and thus locally bounded  with locally bounded derivatives, the lemma  is standard (cf \cite{Golse06}) as soon as the  solution does not explode in finite time. The latter is obvious, since  the quadratic terms are non positive. Indeed, the functions $(y_{1}, y_{2})$  are dominated  by the solution $(z_{1}, z_{2})$ of the system 
\begin{eqnarray*} \label{eq:LVdim2mono} 
\frac{\partial }{\partial t} z_i(t) = r_i(x) z_{i}(t) -\beta_i(x) \lambda_{ii} z_i^2\ ;\ 
z_i(0) = y_i , \quad i=1,2,
\end{eqnarray*}
and we use Lemma \ref{ineglogistique}.

\noindent The flow clearly  satisfies 
\ben
\varphi^{t_{0},y}_{x,1}(t) &=& y_{1} \exp\left(\int_{t_0}^t (r_1(x) - \beta_1(x) \big( \lambda_{11} \varphi^{t_{0},y}_{x,1}(s)+ \lambda_{12}\varphi^{t_{0},y}_{x,2}(s)  \big)ds\right),\\
\varphi^{t_{0},y}_{x,2}(t) &=& y_{2} \exp\left(\int_{t_0}^t (r_2(x) - \beta_2(x) \big( \lambda_{21} \varphi^{t_{0},y}_{x,1}(s) + \lambda_{22} \varphi^{t_{0},y}_{x,2}(s)\big)ds\right).
\een
\end{proof}

\bigskip \noindent 
\noindent The proof of uniqueness will be based on the mild equation satisfied by any solution of (\ref{eq:limit}). Let us consider a function $G$ defined on ${\cal  X} \times \mathbb{R}_{+}^2$ of class $C^1$ on the two last variables and for any $x\in {\cal  X}$, let us define the first-order differential operator
\be
\label{def:L}
{\cal L}G (x,y) := c_1(x,y) \frac{\partial G}{\partial y_1}(x,y) + c_2(x, y) \frac{\partial G}{\partial y_2} (x,y) = c\cdot \triangledown_y G \, (x,y) 
,\ee
where the notation $\cdot$ means the scalar product in $\mathbb{R}^2$. \\
Then the function $\tilde G(s,t,x,y):= G(x,\varphi^{s,y}_{x}(t))$ satisfies
\begin{eqnarray} \label{edp:G}
\frac{\partial }{\partial t} \tilde G & =& 
\triangledown_y G (x, \varphi^{s,y}_{x}(t)) \cdot \frac{\partial }{\partial t} \varphi^{t_{0},y}_{x}(t) \nonumber \\
& =&  c(x, \varphi^{t_{0},y}_{x}(t)) \cdot \triangledown_y G (x,\varphi^{s,y}_{x}(t))\nonumber\\
&=& {\cal L}G (x,\varphi^{s,y}_{x}(t)) = {\cal L} \tilde G
\end{eqnarray} 
Let us fix $t>0$. We deduce from (\ref{edp:G}) and from the flow property \eqref{flow} that $\tilde G$ satisfies the backward transport equation:
\be
\label{eqntransport}
\frac{\partial }{\partial s} \tilde G +  {\cal L}\tilde G = 0, \ \forall s\leq t\ \textrm{ with }\quad \tilde G(t,t,x,y) = G(x, y). \ee

\noindent We now write \eqref{eq:limit} applying the measure $v_t$ to the time-dependent function  $(s,x,y) \mapsto \tilde G(s,t, x, y)$ 
where $G(x, y)=f(x)g(y)$ 
and obtain the mild equation
\begin{align}
\label{eq:mild}
\langle v_t,fg\rangle & =\langle v_0,  f\, g\circ  \varphi^{0,y}_{x}(t)\rangle  +\int_0^t\int_{{\cal X}\times \rit_+^2} \bigg\{\Big(B(x,y_1,y_2)(1-p(x)) \notag \\
& \hskip 3cm -\big(D(x,y_1,y_2)+\alpha(x,y_1,y_2)\ U*v_s(x,y_1,y_2)\big)\Big)
  f(x) g\circ  \varphi^{s,y}_{x}(t)\notag \\
&  + p(x)B(x,y_1,y_2)\int_{}  f(x+z) g\circ \varphi^{s,y}_{x+z}(t) M(x,z)dz \bigg\}
     v_s(dx,dy_1,dy_2)\  ds.
\end{align}
(The last term involving the quantity $\frac{\partial }{\partial s} g\circ \varphi^{s,y}_{x}(t) +  {\cal L}g (\varphi^{s,y}_{x}(t))$ vanishes by \eqref{eqntransport}.)\\

\bigskip \noindent 
Let us now consider two  continuous functions $v$ and $\bar v$ in $C([0,T],M_F({\cal X}\times \mathbb{R}_+^2))$ solutions of \eqref{eq:limit}  with the same initial condition $v_{0}$. Then the difference of both solutions satisfies
\begin{align*}
& \langle v_t - \bar v_{t} ,fg_1g_2\rangle  =
\int_0^t\int_{{\cal  X}\times \rit_+^2} \bigg\{\bigg[\Big(B(x,y_1,y_2)(1-p(x))   -D(x,y_1,y_2)\Big) f(x) g\circ \varphi^{s,y}_{x}(t) \\
&
+ p(x)B(x,y_1,y_2)\int_{} f(x+z) g\circ\varphi^{s,y}_{x+z}(t) M(x,z)dz\bigg]
   (  v_s(dx,dy_1,dy_2)-\bar v_s(dx,dy_1,dy_2))\\
&- \alpha(x,y_1,y_2) f(x) g\circ  \varphi^{s,y}_{x}(t) \Big( U*v_s(x,y_1,y_2) v_s(dx,dy_1,dy_2)- U*\bar v_s(x,y_1,y_2) \bar v_s(dx,dy_1,dy_2)\Big)\bigg\} ds.
\end{align*}
The finite variation norm of a measure $v$ is defined as usual by
$$\|v\|_{FV} := \sup \{ \langle v, h \rangle, h \hbox{ measurable and bounded by } 1\}.$$
Since all coefficients are bounded as well as 
the total masses of $v_{t}$ and $\bar v_{t}$ , it is easy to show that there exists a constant $C_{T}$ such that
\ben
\|v_{t} -\bar v_t\|_{FV} \leq C_{T} \int_{0}^t \|v_{s} -\bar v_s\|_{FV} \, ds,
\een
which implies, by Gronwall's Lemma, that $v$ and $\bar v$ are equal .
\end{proof}

\noindent
Let us now prove that if the initial measure  has a density with respect to  Lebesgue measure, then there exists a unique function solution of \eqref{edptransport}. That gives a general  existence and uniqueness result for such nontrivial equations  with  nonlinear reaction and transport terms, and a nonlocal term involved by the mutation kernel. The existence takes place in a very general set of $L^1$-functions. 

\begin{prop}
\label{prop:density}
Assume that the initial measure $v_{0}$ admits a density $\phi_0$ with respect to the Lebesgue measure $dxdy_{1}dy_{2}$; then for each $t>0$, the measure $v_{t}$ solution of \eqref{edptransport} also admits a density. 
\end{prop}

\begin{proof}
Let us come back to the equation \eqref{eq:mild} satisfied by $v$. Using basic results on linear parabolic equations, we construct by induction a sequence of functions $(\phi^{n})_{n}$ satisfying in a weak sense the following semi-implicit scheme: $\phi^{n+1}_0 \equiv  \phi_{0}$ and

\begin{align}
\label{eq:linmild}
&\langle \phi^{n+1}_t,fg\rangle  =\langle \phi_{0},  f\, g\circ  \varphi^{0,y}_{x}(t)\rangle  +\int_0^t\int_{{\cal X}\times \rit_+^2} \Big[\bigg\{\Big(B(x,y_1,y_2)(1-p(x))  f(x) g\circ  \varphi^{s,y}_{x}(t)  \notag \\
&  \hskip 2cm + p(x)B(x,y_1,y_2)\int_{}  f(x+z) g\circ \varphi^{s,y}_{x+z}(t) M(x,z)dz \bigg\}
     \phi^n_s(x, y_{1},y_{2})\notag \\
     &-\big(D(x,y_1,y_2)+\alpha(x,y_1,y_2)\ U*\phi^n_s(x,y_1,y_2)\big)\Big)
  f(x) g\circ  \varphi^{s,y}_{x}(t)  \phi^{n+1}_s(x, y_{1},y_{2}) 
\Big] dxdy_1dy_2 ds.
\end{align}

\me
Thanks to the nonnegativity of $\phi_{0}$ and of the parameters $B$, $p$, $1-p$, and applying the maximum principle for transport equations (Cf. \cite{Golse06}), we can show that the functions $\phi^n$ are nonnegative.     
   \me  Taking $f=g=1$ and 
     thanks to the nonnegativity of the functions $\phi^{n}$ and to the boundedness of the coefficients  we get
     \ben
   \sup_{s\leq t}  \|\phi^{n+1}_{s}\|_{1} \leq \|\phi_{0}\|_1 +C_1\, \int^t_{0} \sup_{u\leq s}\|\phi^{n}_{u}\|_{1} du,
     \een
     where the constant $C_{1}$ does not depend on $n$ and can be chosen uniformly on [0,T].
   By  Gronwall's Lemma, we conclude that
     \be 
       \label{L1-bound} 
     \sup_{n}\sup_{t\leq T}  \|\phi^{n}_{t}\|_{1}  \leq \|\phi_{0}\|_1\, e^{C_{1}T}.
     \ee
     
     \me Let us now prove the convergence of the sequence $\phi^n$ in $L^{\infty}([0,T], L^1)$. A straightforward computation using \eqref{eq:linmild}, \eqref{L1-bound}, the assumptions on the coefficients and similar arguments as above yields
     \ben
     \sup_{s\leq t}  \|\phi^{n+1}_{s}-\phi^{n}_{s}\|_{1}  \leq  C_{2}\, \int^t_{0} \Big(
      \sup_{u\leq s}  \|\phi^{n+1}_{u}-\phi^{n}_{u}\|_{1} +  \sup_{u\leq s}  \|\phi^{n}_{u}-\phi^{n-1}_{u}\|_{1}\Big) ds,
      \een
    where $C_{2}$ is a positive constant independent of $n$ and $t\in[0,T]$. It follows from Gronwall's Lemma that for each $t\leq T$ and $n$,
    $$   \sup_{s\leq t}  \|\phi^{n+1}_{s}-\phi^{n}_{s}\|_{1}  \leq  C_{3}\, \int^t_{0} \sup_{u\leq s}  \|\phi^{n}_{u}-\phi^{n-1}_{u}\|_{1}\, ds.$$
    We conclude that the series $\ \sum_{n}\sup_{t\in [0,T]}  \|\phi^{n+1}_{t}-\phi^{n}_{t}\|_{1} \ $ converges  for any $T>0$. Therefore the sequence of  functions $(\phi^n)_{n}$ converges in $L^{\infty}([0,T], L^1)$ to a continuous function $t \mapsto \phi_t$  satisfying
    $$\sup_{t\leq T}  \|\phi_{t}\|_{1}  \leq \|\phi_{0}\|_1\, e^{C_{1}T}.$$
    Moreover, since the sequence converges in $L^1$, the limiting  measure $\phi_{t}(x,y_{1},y_{2}) dx dy_{1}dy_{2}$ is solution of \eqref{eq:mild} and then it is its unique solution. 
    Hence, that implies that for all $t$, 
    $$v_{t}( dx, dy_{1},dy_{2})= \phi_{t}(x,y_{1},y_{2})\,  dx dy_{1}dy_{2}. $$
       \end{proof}
       
  \me We have thus proved that  the  nonlinear integro-differential equation   \eqref{edptransport} admits a unique weak function-valued solution as soon as the initial condition $\phi_0$ is a $L^1$-function, without any additional regularity assumption.

%%%%%%%%%%%%%%%%%%%%%%%%%%%%%%%%%%%%%%%%%%%%%%%%%%%%%%%%%%%%%%%%%%%%%%%%%%%%%%%%%%%%%%%%%
\subsection{Stationary states under a mean field assumption and without trait mutation}
%%%%%%%%%%%%%%%%%%%%%%%%%%%%%%%%%%%%%%%%%%%%%%%%%%%%%%%%%%%%%%%%%%%%%%%%%%%%%%%%%%%%%%%%

This part is a first step in the research of  stationary states for the deterministic measure-valued process $(v_t, t\geq 0)$ defined above. 
We firstly remark that equation (\ref{totalmass}), which determines the evolution 
of the total number of individuals $t
\mapsto  \langle v_t,1 \rangle$, is not closed if the functions $U, B, D$ or $\alpha$ are not constant, which makes the problem very hard. 
In this section we consider the simplest case where  the individual ecological
parameters $B$ and $D$ and 
the cell ecological parameters $b_i$ and $d_i$ are constant and where the mutation probability $p$ vanishes.
Moreover, we work under the mean field assumption, 
that is the competition/selection kernel $U$ is a constant. 
We consider two different cases corresponding to different selection rates $\alpha$.

%%%%%%%%%%%%%%%%%%%%%%%%%%%%%%%%%%%%%%%%%%%%%%%%%%%%%%%%
\subsubsection{Case with constant selection rate}
%%%%%%%%%%%%%%%%%%%%%%%%%%%%%%%%%%%%%%%%%%%%%%%%%%%%%%%%

Let us  assume  that  the selection rate $\alpha$ is constant. 
In this case, the mass equation (\ref{totalmass}) is closed and reduces to the
standard logistic equation
\begin{align}
\label{closemass} \langle v_t,1\rangle  =\langle 
v_0,1\rangle  +\int_0^t  \langle v_s,1\rangle \big((B - D) -\alpha
U\langle v_s,1\rangle \big)\ ds,
\end{align}
whose asymptotical behavior is well known: the mass of 
any stationary measure $v_{\infty}$ satisfies
$$
(B - D)\langle v_{\infty},1\rangle =\alpha U\langle v_{\infty},1\rangle^2.
$$
Either $R:= B-D \leq 0$ and there is extinction of the population, that is 
$$\lim_{t \rightarrow +\infty}\langle v_t,1\rangle = \langle v_{\infty},1\rangle = 0.
$$
Or $R>0$ and the mass of the population converges to a non degenerate value
\be \label{stat} 
\lim_{t \rightarrow +\infty}\langle v_t,1\rangle = \langle v_{\infty},1\rangle = \frac{R}{\alpha U} .
\ee
Furthermore, the convergence of the mass holds exponentially fast: 
due to (\ref{closemass}), 
$$ \frac{\partial }{\partial t} \langle v_t - v_{\infty},1 \rangle= - \alpha U \langle v_t,1\rangle \langle v_t - v_{\infty},1 \rangle .$$
\be
\label{vitconvmasse}
\textrm{Thus }\qquad \langle v_t - v_{\infty},1 \rangle = \langle v_0 - v_{\infty} , 1 \rangle \, e^{-\alpha U \int_0^t \langle v_s,1 \rangle \ ds} 
\ee
which vanishes exponentially fast. 

\medskip
\noindent Assume $R>0$ in such a way that the mass of the population does not vanish.
 In what follows we will need the following notations:
$$
\langle \overline{v,1} \rangle := \sup_t \langle v_t,1 \rangle < +\infty
$$
and 
$$
\bar\alpha := \sup_t \alpha_t (< +\infty) \quad \textrm{ where } \quad \alpha_t := R-\alpha U \langle v_t,1\rangle = - \alpha U \langle v_t - v_{\infty},1\rangle 
.
$$

\noindent Let us now consider the weak convergence of the measures $v_t$ towards the stationary measure $v_\infty$, which is concentrated on the equilibrium state of the Lotka-Volterra dynamics.

\noindent Applying equation (\ref{eq:limit}) to any bounded smooth fonction $g(y)=g_1(y_1)g_2(y_2)$,
\begin{eqnarray} \label{eq:vtg}
\frac{\partial }{\partial t}\langle v_t,g \rangle & =&
\alpha_t \langle v_t,g \rangle + \langle v_t,r_1 y_1 \frac{\partial g }{\partial y_1} + r_2 y_2 \frac{\partial g }{\partial y_2} \rangle \nonumber\\
&&- \langle v_t, \beta_1 \frac{\partial g }{\partial y_1}\big( \lambda_{11} y_1 + \lambda_{12} y_2\big) y_1
+ \beta_2 \frac{\partial g }{\partial y_2} \big( \lambda_{21} y_1 + \lambda_{22} y_2\big) y_2
\rangle \nonumber\\
& = & \alpha_t \langle v_t,g \rangle + \langle v_t,{\cal L}g \rangle
\end{eqnarray} 
where the differential first order operator ${\cal L}= c \cdot \triangledown  $ is the same as in (\ref{def:L}) but without dependence on the trait $x$.
Using the flow of Lotka-Volterra equation (see (\ref{eq:LVdim2})), we represent the mild solution of (\ref{eq:vtg}) as

\begin{equation}
\label{eq:vtgmild}
\langle v_t,g \rangle = \int_{\mathbb{R}^2_+}  g\circ \varphi^{0,y }(t) \, v_0(dy) +
\int_0^t  \alpha_s  \int_{\mathbb{R}^2_+} g\circ \varphi^{s,y }(t)  \,  v_s (dy) \, ds .
\end{equation}

\noindent Let us firstly recall the long time behavior of the Lotka-Volterra system \eqref{eq:LVdim2}  in case where the coefficients $c_{i}$ don't depend on $x$(see Istas \cite{Istas05}).
\begin{lem}
\label{proporcells}
\label{LTBLV}
Any solution of 
\begin{eqnarray} \label{eq:LVdim2const} 
\frac{\partial }{\partial t} y_1(t)&=&  y_1(t) \left(r_1 - \beta_1 \big( \lambda_{11} y_1(t) + \lambda_{12} y_2(t) \big)\right) \nonumber\\ 
\frac{\partial }{\partial t} y_2(t)&=&  y_2(t)\left(r_2  - \beta_2 \big( \lambda_{21} y_1(t) + \lambda_{22} y_2(t) \big) \right)
\end{eqnarray} 
with non-zero initial condition in ${\mathbb R}_{+}^2$  converges for $t$ large  to a finite limit, called equilibrium and denoted by $\pi=(\pi_{1}, \pi_{2})\in {\mathbb R}_{+}^2\setminus \{(0,0)\}$. It takes the following values:
\begin{enumerate}
\item $\pi=({r_{1}\over \beta_{1}\lambda_{{11}}},0)\quad $ if $\quad r_{2}\lambda_{{11}}- r_{1}\lambda_{21}< 0$ (resp. $=0$ and $r_{1}\lambda_{{22}}- r_{2}\lambda_{12}>0$).

\item   $\pi=(0,{r_{2}\over \beta_{2}\lambda_{22}}) \quad $ if $\quad r_{1}\lambda_{{22}}- r_{2}\lambda_{12}<0$ (resp. $=0$ and $r_{2}\lambda_{11}- r_{1}\lambda_{21}>0$ ).

\item If  $\quad r_{2}\lambda_{{11}}- r_{1}\lambda_{12}>0$ and $r_{1}\lambda_{{22}}- r_{2}\lambda_{21}>0$  
\begin{equation}
\label{equ:proporcells}
\pi=
\Big(\frac{\beta_1 \lambda_{12}(b_2-d_2)- \beta_2 \lambda_{22}(b_1-d_1)}{\beta_1\beta_2 (\lambda_{12}\lambda_{21}-\lambda_{11}\lambda_{22})},\frac{\beta_2 \lambda_{21}(b_1-d_1)- \beta_1 \lambda_{11}(b_2-d_2)}{\beta_1\beta_2 (\lambda_{12}\lambda_{21}-\lambda_{11}\lambda_{22})} \Big) .
\end{equation}

\end{enumerate}

\end{lem}
Therefore we obtain the following convergence result.
\begin{prop}
 \label{convversmestat}
The deterministic measure-valued process $v_t$ converges for large time $t$ - in the weak topology - towards the singular measure 
concentrated on the equilibrium state $ \pi$ of the associated Lotka-Volterra dynamics:
$$
\lim_{t \rightarrow +\infty}  v_t =   \frac{R}{\alpha U} \,\delta_{(\pi_1,\pi_2)} ,
$$
where $\pi=(\pi_1,\pi_2)$ is defined in Lemma  \ref{proporcells}.
\end{prop}
\begin{proof}
First,  the Lotka-Volterra flow $\varphi^{0,y }(t)$ converges for large $t$  towards $(\pi_1,\pi_2)$ given by
Lemma \ref{proporcells}. Since the test function $g$ is continuous and bounded and $v_0$ has a finite mass,  Lebesgue's dominated theorem implies that the first term in the right hand side of (\ref{eq:vtgmild}) converges:
$$
\lim_t \int_{\mathbb{R}^2_+}  g\circ \varphi^{0,y }(t) \, v_0(dy) = \int_{\mathbb{R}^2_+}  \lim_t g\circ \varphi^{0,y }(t) \, v_0(dy) = 
g(\pi_1,\pi_2) \, \langle v_0,1 \rangle .
$$
Secondly, as already seen in (\ref{vitconvmasse}), the mass $\langle v_t,1 \rangle $ of the total population converges exponentially fast to its equilibrium size, that is  $\alpha_t $ converges exponentially fast to 0: 
$$
\exists c>0, \exists t_0, \quad \forall s>t_0 \quad  \alpha_s \leq e^{-cs} .
$$
Therefore the second term in the right hand side of (\ref{eq:vtgmild}) can be disintegrated, for $t$ larger than $t_0$,  in the sum of two integrals over $[0,t_0]$ and $[t_0,t]$.
The control of the integral over $[t_0,t]$ is simple: 
$$
\left| \int^t_{t_0}  \alpha_s  \int_{\mathbb{R}^2_+} g\circ \varphi^{s,y }(t)  \,  v_s (dy) \, ds \right| \leq \langle \overline{v,1} \rangle \sup_y |g(y)|  \int^t_{t_0} e^{-cs} ds
$$
which is as small as one wants, when $t_0$ is large enough.\\
On the compact time interval $[0, t_0]$ the following convergence holds:
\begin{eqnarray*} 
\lim_t \int_0^{t_0}  \alpha_s  \int_{\mathbb{R}^2_+} g\circ \varphi^{s,y }(t)  \,  v_s (dy) \, ds &=& 
\int_0^{t_0}  \alpha_s  \int_{\mathbb{R}^2_+} \lim_t g\circ \varphi^{s,y }(t)  \,  v_s (dy) \, ds \\
&=& 
g(\pi_1,\pi_2)  \int_0^{t_0}  \alpha_s  \int_{\mathbb{R}^2_+} v_s (dy) \, ds .
\end{eqnarray*} 
Therefore for large time $t>t_0$, $\langle v_t,g \rangle $ is as close as one wants to
$$
g(\pi_1,\pi_2) \, \langle v_0,1 \rangle \, +  g(\pi_1,\pi_2)\,  \int_0^{t_0}  \alpha_s  \int_{\mathbb{R}^2_+} v_s (dy) \, ds   = g(\pi_1,\pi_2) \, \langle v_{t_0},1 \rangle .
$$
For $t_0$ large enough, this last quantity is close to $g(\pi_1,\pi_2) \, \langle v_\infty,1 \rangle  = \frac{R}{\alpha U} \,\langle \delta_{(\pi_1,\pi_2)}, g \rangle $.

\noindent This completes the proof of the weak convergence of the measures $v_t$.
\end{proof}

\begin{rema} 

The  stationary state is  a singular one 
even if the initial measure $v_0$ has a density: the absolute continuity property 
of the measure $v_t$ is conserved for any finite time $t$, but it is lost in infinite time.
\end{rema}

%%%%%%%%%%%%%%%%%%%%%%%%%%%%%%%%%%%%%%%%%%%%%%%%%%%%%
\noindent {\bf Convergence of the number of cells}\\
%%%%%%%%%%%%%%%%%%%%%%%%%%%%%%%%%%%%%%%%%%%%%%%%%%%%%
\noindent First we prove the boundedness of the number of cells of each type and the boundedness of its second moment. 
To this aim, we compare the multitype dynamics with a  dynamics where the different types do not interact, which corresponds to two independent monotype systems.

\begin{lem} \label{momentnbrecellulesmulti}
If $\langle v_0  , 1 \rangle + \langle v_0  , y_i^2 \rangle < + \infty$ 
then 
$\sup_{t \geq 0} \, \langle v_t  , y_i^2 \rangle < + \infty $. 
\end{lem}
\begin{proof}
Let us firstly prove that $\sup_{t \geq 0} \, \langle v_t  , y_i \rangle < + \infty $. \\
At time $t=0$ , $\langle v_0  , y_i \rangle \leq \langle v_0 , 1 \rangle+ \langle v_0  , y_i^2 \rangle < + \infty. $ Moreover, 
equation (\ref{cellmassgen}) reads now 
\begin{eqnarray} \label{ctcellmass} 
\frac{\partial }{\partial t}\langle v_t,y_i \rangle & =& 
\big( R-\alpha U\langle v_t,1\rangle \big)\langle v_t,y_i\rangle 
+ (b_{i}-d_i)\langle v_t,y_i\rangle
-\beta_{i}(\lambda_{ii}\langle v_t,y_i^2\rangle+\lambda_{ij}\langle v_t,y_iy_j\rangle )\nonumber\\
&\leq &( \alpha_t +b_{i}-d_i) \langle v_t,y_i\rangle - \beta_i\lambda_{ii} \langle v_t,y_i^2\rangle \nonumber\\
&\leq &( \alpha_t +b_{i}-d_i) \langle v_t,y_i\rangle - \frac{\beta_i\lambda_{ii}}{\langle v_t,1 \rangle } \langle  v_t,y_i \rangle^2 \nonumber\\
&\leq & ( \bar\alpha +r_i) \langle  v_t,y_i\rangle - \frac{\beta_i\lambda_{ii}}{\langle \overline{v,1} \rangle } \langle  v_t,y_i\rangle^2.
\end{eqnarray}
This inequality is a logistic one in the sense of 
Lemma \ref{ineglogistique}. Therefore one  deduces that the number of cells of type $i$ is  uniformly bounded in time:
$$
\sup_{t \geq 0} \langle v_t,y_i \rangle  < + \infty , \quad i =1,2 .
$$
By (\ref{eq:limit}) applied with $f\equiv 1, \, g_1(y_1)=y^2_1, \, g_2\equiv 1$, one obtains
\begin{eqnarray*} 
\frac{\partial }{\partial t}\langle v_t,y_1^2 \rangle & =&
\alpha_t \langle v_t,y_1^2 \rangle + 2r_1 \langle v_t,y_1^2 \rangle  
- 2 \beta_1 \Big( \lambda_{11} \langle v_t,y_1^3\rangle + \lambda_{12}\langle v_t,y_1^2 y_2\rangle \Big)\\
& \leq &
 ( \alpha_t + 2r_1) \langle v_t,y_1^2 \rangle  - 2 \beta_1 \lambda_{11} \langle v_t,y_1^3\rangle \\
& \leq &
 ( \alpha_t + 2r_1) \langle v_t,y_1^2 \rangle  - 2 \beta_1 \lambda_{11}\frac{1}{\langle v_t,y_1 \rangle } \langle v_t,y_1^2\rangle^2 \\
& \leq &
 ( \bar\alpha + 2r_1) \langle v_t,y_1^2 \rangle  - 2 \beta_1 \lambda_{11}\frac{1}{\langle \overline{v,y_1} \rangle } \langle v_t,y_1^2\rangle^2 
    \end{eqnarray*} 
since
$$
\langle v_t,y_1^2 \rangle^2 \leq 
\langle v_t,y_1^3\rangle  \langle v_t,y_1 \rangle .
$$
This inequality on  $\langle v_t,y_1^2 \rangle $ is of logistic type  as  (\ref{ctcellmass}). Lemma \ref{ineglogistique} implies 
$$
\langle \overline{v,y_1^2} \rangle:= \sup_{t \geq 0} \langle v_t,y_1^2 \rangle < + \infty.
$$
The same holds for $\langle \overline{v,y_2^2} \rangle$.
\end{proof}

\begin{prop} \label{statmeasurebitype}
If  $\langle v_0  , y_i \rangle < + \infty$  and
$\langle v_0  , y_i^2 \rangle < + \infty$, 
then 
the total number of cells of each type per individual $\frac{\langle v_t,y_i \rangle }{\langle v_t,1 \rangle} $ stabilizes for $t$ large:
$$
\lim_{t \rightarrow + \infty} \frac{\langle v_t,y_i \rangle }{\langle v_t,1 \rangle}  = 
\pi_i .
$$ 
\end{prop}

\bigskip
\vspace{0.2cm}
\begin{proof}
Due to Proposition \ref{convversmestat}, the family of measures  $(v_t)_t$ converge weakly towards $v_\infty$. Moreover, 
by Lemma \ref{momentnbrecellulesmulti}, the second moments of  $v_t$ are uniformly bounded.
Therefore $y_i$ is uniformly integrable under the family of $(v_t)_t $ which leads to :
\begin{eqnarray*}
\lim_{t \rightarrow +\infty} \langle v_t,y_i \rangle & = & \langle \lim_{t \rightarrow +\infty} v_t,y_i \rangle = \langle v_\infty,y_i \rangle .
\end{eqnarray*}
\end{proof}

\medskip 

\noindent Let us  underline the decorrelation at infinity between cell and individual dynamics.

%%%%%%%%%%%%%%%%%%%%%%%%%%%%%%%%%%%%%%%%%%%%%%%%%%%%%%%%%%%%%%%%
\subsubsection{Case with linear selection rate}
\label{sec:alphalineaire}
%%%%%%%%%%%%%%%%%%%%%%%%%%%%%%%%%%%%%%%%%%%%%%%%%%%%%%%%%%%%%%%%

Suppose now that the selection rate $\alpha$ does not depend on the trait $x$ 
but is linear as function of the number of cells of each type~:
$$
\exists \alpha_1, \alpha_2 \in ]0,1[, \quad  \alpha (x,y_1,y_2) = \alpha_1 y_1 + \alpha_2 y_2 =:\alpha \cdot y.
$$
With other words the selection increases linearly when the number of cells increases.\\
The new main difficulty  comes from the fact that the  mass equation is no more closed~:
\be \label{eqmassalineaire}
\langle v_t,1\rangle  =\langle v_0,1\rangle  + \int_0^t \langle v_s,1\rangle \big( R -
U\langle v_s,\alpha \cdot y \rangle \big)\ ds,
\ee
which  has as (implicit) solution
\begin{equation}
\label{massalineaire}
\langle v_t,1 \rangle = \langle  v_0,1 \rangle e^{ - \int_0^t \big(  U\langle v_s,\alpha \cdot y \rangle -R \big) ds}.
\end{equation}

\noindent For this reason, unfortunately, we did not succeed in proving  the convergence in time of 
$\langle v_t,1\rangle $. 
Nevertheless, we can conjecture some limiting behavior of the process.\\

\noindent {\bf Conjecture} : 
The deterministic measure-valued process $v_t$ converges for large time $t$ towards the following stationary value
\begin{equation}\label{statmeasurealineaire}
\lim_{t \rightarrow +\infty}  v_t =   v_\infty := \frac{R}{U (  \alpha_1\pi_1 +  \alpha_2\pi_2 )} \, \delta_{(\pi_1,\pi_2)}  ,
\end{equation}
where $\pi=(\pi_1,\pi_2)$ is given in Lemma  \ref{proporcells}.\\
\medskip

\noindent  In this case too, the asymptotic proportions of the cells of different types per individual would become deterministic and independent.

\medskip

\noindent {\bf Some partial answers} 
\begin{itemize}
\item 
Equation (\ref{eqmassalineaire}) implies that any stationary measure $v_{\infty}$  should satisfy
$$
\langle v_{\infty},1\rangle \big( R- U \langle v_{\infty},\alpha \cdot y \rangle  \big)=0.
$$
Then, either $\langle v_{\infty},1\rangle = 0$, that means the extinction of the
individual population holds, or 
\be \label{statalineaire} 
\langle v_{\infty},\alpha \cdot y \rangle = \langle v_{\infty},\alpha_1 y_1 + \alpha_2 y_2 \rangle =\frac{B-D}{ U} =\frac{R}{ U} 
\ee
which describes a constraint between the limiting number of the different types of cells.
\item {\it Boundedness of the number of cells.}
\begin{lem} \label{nbrecellulesalineairemono}
$$\langle v_0  , y_i \rangle < + \infty \quad \Longrightarrow   \quad
\sup_{t\geq 0} \, \langle v_t,y_i \rangle < + \infty .$$
\end{lem}
\begin{proof}
The number of cells of type $i$ satisfies 
\begin{eqnarray}
\frac{\partial}{\partial t} \langle v_t,y_i\rangle & = & 
\langle v_t,y_i \rangle  \big( R+ r_i\big) - U \langle v_t,y_i \alpha \cdot y \rangle \langle v_t,1\rangle 
-\beta_i \big( \lambda_{ii}\langle v_t,y_i^2\rangle + \lambda_{ij} \langle v_t,y_iy_j \rangle \big)\nonumber \\
&\leq &    ( R +r_i) \langle  v_t,y_i\rangle - \alpha_i U \langle v_t,y_i\rangle^2 
\end{eqnarray}
which reduces to the monotype case solved in Lemma \ref{ineglogistique}.
\end{proof}
\item {\it Identification of  a unique possible non trivial equilibrium.}\\

Applying Equation (\ref{eq:limit}) to  $f\equiv 1, \, g_i(y)=e^{-z_iy}$ and letting $t$ tend to infinity, we remark that 
the Laplace transform ${\bf L}_{\infty}(z)$ of any non vanishing stationary state $ v_{\infty}$ should satisfy 
\begin{align}
\label{edpalineaire}
& R {\bf L}_{\infty}(z) - U\langle v_{\infty},1\rangle \langle v_{\infty},(\alpha \cdot y)e^{-z\cdot y}\rangle \notag \\
&  - \sum_{i=1}^2 z_i \Big( r_{i}\langle v_{\infty},y_i e^{-z\cdot y}\rangle
-\beta_{i}\big(\lambda_{ii}\langle v_{\infty},y_i^2e^{-z\cdot y}\rangle+\lambda_{ij}\langle v_{\infty},y_iy_je^{-z\cdot y} \rangle\big)
\Big) = 0  \notag \\
&\Rightarrow 
R {\bf L}_{\infty} + 
U {\bf L}_{\infty}(0)(\alpha_1\frac{\partial {\bf L}_{\infty}}{\partial z_1}+ \alpha_2 \frac{\partial {\bf L}_{\infty}}{\partial z_2}) 
\notag \\  
& \qquad + \sum_{i=1}^2 \Big( z_i r_{i}\frac{\partial {\bf L}_{\infty}}{\partial z_i} +
z_i \beta_{i}\big(\lambda_{ii}\frac{\partial^2 {\bf L}_{\infty}}{\partial z_i^2}+
\lambda_{ij}\frac{\partial^2 {\bf L}_{\infty}}{\partial z_i\partial z_j}\big) \Big)= 0
\end{align}
with usual boundary conditions
$$
{\bf L}_{\infty}(0)= \langle v_{\infty},1 \rangle,  \quad \frac{\partial {\bf L}_{\infty}}{\partial z_i}(0) = -\langle v_{\infty},y_i\rangle .
$$
The unique non trivial solution of this p.d.e. is 
$$
{ \bf L}_{\infty}(z) = \langle v_{\infty},1 \rangle \, e^{- \tilde \pi \cdot z},
$$
where $\langle v_{\infty},1 \rangle = \frac{U}{\alpha\cdot \tilde\pi}$ and where $\tilde\pi_i$, the equilibrium proportion of cells of type $i$ in the global population,
has to be equal
to the equilibrium proportion given in 
Lemma  \ref{proporcells}: $\tilde\pi=\pi$.

%%%%%%%%%%%%%%%%%%%%%%%%%%%%%%%%%%%%%%%%%
\item  {\it Local stability of the non trivial equilibrium $v_\infty:= \frac{R}{U  \alpha\cdot \pi} \, \delta_{(\pi_1,\pi_2)} $.}\\
%%%%%%%%%%%%%%%%%%%%%%%%%%%%%%%%%%%%%%%%%%%%%%%%%%%%%%%%%%%%%

Although we cannot control the convergence of $\langle v_t,1\rangle $ to a positive number, 
we can analyze the stability of the nontrivial stationary state $v_\infty$ in the following sense. \\

{\it Stability of the mass around its positive stationary value $\frac{R}{U  \alpha\cdot \pi}$}\\
Let start with $v_0 = v_\infty + \varepsilon \delta_{(\zeta_1,\zeta_2)} $, where $\varepsilon$ is small and $(\zeta_1,\zeta_2) \in \rit_+^2$. 
From the mass equation ({\ref{eqmassalineaire}}) one obtains for $t$ small~:
\begin{eqnarray*} 
\frac{\partial}{\partial t} \langle v_t,1\rangle & = &\langle v_t,1\rangle \big( R -U \langle v_t,\alpha \cdot y \rangle \big)\\
& \simeq  & (\langle v_\infty,1\rangle + \varepsilon ) \big( R -U \langle v_\infty,\alpha \cdot y \rangle -\varepsilon U \,  \alpha \cdot \zeta \big)\\
& \simeq  &  -\varepsilon \, \frac{\alpha \cdot \zeta}{\alpha\cdot \pi}  + o(\varepsilon ).
\end{eqnarray*}
This quantity is negative for small $\varepsilon $, which implies the stability of the mass around its positive stationary value.\\

{\it Stability of the number of cells of each type around its limit value if $\max (r_1, r_2 )< R$}\\
We prove it only for the type $1$. From (\ref{cellmassgen}) we get an expansion in $\varepsilon$ of the variation of the global number of cells of type $1$ for small time~: 
\begin{eqnarray*}
\frac{\partial}{\partial t} \langle v_t,y_1\rangle & = & 
\langle v_t,y_1 \rangle  \big( R + r_1 \big) - U \langle v_t,y_1 \,\alpha \cdot y \rangle \langle v_t,1\rangle 
-\beta_1 \big( \lambda_{11}\langle v_t,y_1^2\rangle + \lambda_{12} \langle v_t,y_1y_2 \rangle \big)\\
& \simeq  &  \big( \langle v_\infty,y_1\rangle +\varepsilon \zeta_1 \big) \big( R + r_1 \big) 
- U \big( \langle v_\infty,y_1 \,\alpha \cdot y \rangle +\varepsilon  \zeta_1 \,\alpha \cdot \zeta \big) (\langle v_\infty,1\rangle +\varepsilon )\\
&& \quad -\beta_1 \big( \lambda_{11}\langle v_\infty,y_1^2\rangle +\lambda_{12} \langle v_\infty,y_1y_2 \rangle \big)
-\varepsilon \beta_1 ( \lambda_{11}\zeta_1^2 +\lambda_{12}\zeta_1\zeta_2 )  \\
& = &  \varepsilon  \Big( ( R + r_1) \zeta_1 
- U \langle v_\infty,y_1 \,\alpha \cdot y \rangle -U \zeta_1 \,\alpha \cdot \zeta \langle v_\infty,1\rangle 
-\beta_1 ( \lambda_{11}\zeta_1^2 +\lambda_{12}\zeta_1\zeta_2) \Big)+ o(\varepsilon )\\
& = &  \bar P_1(\zeta_1,\zeta_2)+ o(\varepsilon )
\end{eqnarray*}
where $\bar P_1(y_1,y_2) \leq P_1(y_1)$ for all $y_2>0$, with 
$$
P_1(X):= -\big(U \alpha_1 \langle v_\infty,1\rangle + \beta_1 \lambda_{11} \big) X^2
+ ( R + r_1) X -  U \alpha_1 \pi_1^2 \langle v_\infty,1\rangle .
$$
As second degree polynomial $P_1$ is negative if its discriminant is non positive. 
This condition is fulfilled when
$$
( R + r_1)^2 -4 U \alpha_1 \pi_1^2 \langle v_\infty,1\rangle \big(U \alpha_1 \langle v_\infty,1\rangle + \beta_1 \lambda_{11} \big) <0.
$$
It is true as soon as 
$$
(R + r_1)^2 -4 R^2  <0 \Leftrightarrow r_1 < R .
$$
Thus if $\max (r_1, r_2 )< R$, the number of cells of each type is stable around its limiting value.
\end{itemize}

%%%%%%%%%%%%%%%%%%%%%%%%%%%%%%%%%%%%%%%%%%%%%%%%%%%%
%%%%%%%%%%%%%%%%%%%%%%%%%%%%%%%%%%%%%%%%%%%%%%%%%%%%
\section{Diffusion and superprocess approximations}
\label{sec:limit3}
%%%%%%%%%%%%%%%%%%%%%%%%%%%%%%%%%%%%%%%%%%%%%%%%%%%%

As in the above section we introduce the renormalization $\kappa= (K, K_1, K_2)$ both for individuals and for cells. Moreover we introduce an acceleration of individual births
and deaths with a factor $K^\eta$ (and a  mutation kernel  $M_K$ with amplitude of order  $K^{\eta/2}$ ) 
and  an acceleration of cell births
and deaths with a factor $K_1$ (resp. $K_2$).

\me We summarize below the assumptions we need on the model and which will be considered in all this section. 

\newpage
\me {\bf Assumptions (H3)}: \\
{\it 1) There exist continuous functions $\Gamma, B, D, \alpha$ on ${\cal X} \times \rit_+\times \rit_+$ 
such that
\be \label{scalingbis}
&& B_\kappa(x,n_1,n_2)= K^{\eta}\, \Gamma(x, \frac{n_1}{K_1} ,\frac{n_2}{K_2}) + B(x, \frac{n_1}{K_1} ,\frac{n_2}{K_2});\nonumber\\
&& D_\kappa(x,n_1,n_2)= K^{\eta}\, \Gamma(x, \frac{n_1}{K_1},\frac{n_2}{K_2}) + D(x, \frac{n_1}{K_1} ,\frac{n_2}{K_2});\nonumber\\
&&\alpha_\kappa(x, n_1,n_2)\equiv \alpha(x,\frac{n_1}{K_1},\frac{n_2}{K_2}). 
\ee
The function $\Gamma$ is assumed to be bounded and $B, D, \alpha$ satisfy Assumptions (H1).  \\
2) As before, the competition kernel satisfies 
$$ U_\kappa(x)={U(x)\over K},$$
where $U$ is a continuous function which satisfies Assumption (H1).  \\
3) The mutation law $ z\mapsto  M_K(x,z)$ is a centered  probability density  on ${\cal X} - x$.  Its covariance matrix is  $\frac{\sigma(x)^2}{K^\eta} Id$, where $\sigma $ is a continuous function. We also assume that $$\lim_{K\to \infty} K^\eta\, \sup_{x} \int |z|^3 M_{K}(x,z)dz = 0.$$
The parameter $p_\kappa$ stays unchanged: $p_\kappa(x)=p(x)$ .\\
4)
At the cell level, we introduce Lipschitz continuous functions $b_i, d_i$ on ${\cal X}$  and  a continuous function $ \gamma$ such that \be 
\label{scalingter}
&&b_{i,\kappa}(x)= K_i \,\gamma(x) + b_i(x) ;\nonumber\\
&& d_{i,\kappa}(x)= K_i\, \gamma(x) + d_i(x) , \quad i =1,2.
\ee
The interaction between the cells is  rescaled according on their type~:
\be \label{scalinglambda}
\lambda_{ij}^\kappa=\frac{\lambda_{ij} }{K_j}, \qquad i,j \in \{1,2\} .
\ee
The other parameters stay unchanged:
$\beta_{1,\kappa}=\beta_1$, $\beta_{2,\kappa}=\beta_2$.\\
\me
5) {\bf Ellipticity:} The functions $p$, $\sigma$, $\gamma$ and $\Gamma$ are lower bounded by positive constants 
 and $\sigma\,\sqrt{p\,\Gamma}$ and $\sqrt{\gamma}$ are Lipschitz continuous. \\
}

\noindent As in the section 3, we define the measure-valued Markov process $(Z^\kappa_t)_{t\geq 0}$ as
\begin{equation*}
 Z^\kappa_t= \frac{1}{K} \sum_{i=1}^{I_\kappa(t)} \delta_{(X^i_\kappa(t),\frac{N_{1,\kappa}^i(t)}{K_1},\frac{ N_{2,\kappa}^i(t)}{K_2})}.
\end{equation*}
We may summarize as in Proposition \ref{YK} the moment and martingale properties of $Z^\kappa$.

\newpage
\begin{prop}
\label{YK1}
Assume that for some $p\geq 3$,  $\mathbb{E}(\langle Z^\kappa_0,1\rangle^p + \langle Z^\kappa_0,y_1^2+y_2^2\rangle)<+\infty$. Then
\begin{description}
\item[\textmd{(1)}] 
For any $T>0$, $\mathbb{E}\left(\sup_{t\in[0,T]}\langle Z^\kappa_t,1\rangle^3 + \sup_{t\in[0,T]}\langle Z^\kappa_t,y_1^2+y_2^2\rangle \right)<+\infty.$
\item[\textmd{(2)}] 
For any measurable bounded functions $f, g_1, g_2$, the process
\begin{align}  \label{pbmZK}
& \bar M_t^{\kappa,fg}  = \left<Z^\kappa_t,fg_1g_2\right> - \left<Z^\kappa_0,fg_1g_2\right>  \notag \\
& - \int_0^t \int_{{\cal X} \times \rit^2_+}
   \bigg\{\Big( B(x,y_1,y_2) - D(x,y_1,y_2) - \alpha(x,y_1,y_2) \ U*Z^\kappa_s(x,y_1,y_2)  \Big) f(x)g_1(y_1)g_2(y_2)  \notag \\
& \qquad + p(x)\big( K^{\eta} \, \Gamma(x, y_1 ,y_2) + B(x,y_1,y_2) \big) \int \big( f(x+z)-f(x) \big) g_1(y_1)g_2(y_2)M_K(x,z)dz  \notag \\
& \qquad + f(x) \big( g_1(y_1+\frac{1}{K_1})-g_1(y_1)\big)g_2(y_2)\, \left(K_1 \gamma(x) + b_1(x)\right) \,  K_1 y_1\notag \\
& \qquad + f(x)g_1(y_1)\big(g_2(y_2+\frac{1}{K_2})-g_2(y_2)\big)\, \left(K_2 \gamma(x) + b_2(x)\right) \, K_2 y_2\notag \\
& \qquad + f(x)\big(g_1(y_1-\frac{1}{K_1})-g_1(y_1) \big) g_2(y_2)
 \left( K_1 \gamma(x) + d_1(x) +\beta_{1}(x)(y_1\lambda_{11}+y_2\lambda_{12})\right) K_1 y_1\notag \\
& \qquad + f(x)g_1(y_1)\big(g_2(y_1-\frac{1}{K_2})-g_2(y_2)\big)
  \left( K_2 \gamma(x) + d_2(x) +\beta_{2}(x)(y_1\lambda_{21}+y_2\lambda_{22})\right) K_2 y_2\bigg\} \notag \\
& \hskip 9cm Z_s^\kappa(dx,dy_1,dy_2)\ ds
\end{align}
is a  c\`adl\`ag square integrable $({\cal F}_t)_{t\geq 0}$-martingale 
with quadratic variation
\begin{align}
& \langle \bar M^{\kappa,fg}\rangle_t = 
{1\over K}\intot \int_{{\cal X} \times \rit^2_+}
\bigg\{\Big( 2K^\eta \,\Gamma(x,y_1,y_2) + B(x,y_1,y_2) \notag \\
&\hskip 3cm + D(x,y_1,y_2) + \alpha(x,y_1,y_2) \ U*Z^\kappa_s(x,y_1,y_2) \Big)
f^2(x)g_1^2(y_1)g_2^2(y_2) \notag \\
& + p(x)\big( K^{\eta} \,\Gamma(x, y_1 ,y_2) + B(x,y_1,y_2) \big) \int_{}\big( f(x+z)-f(x) \big)^2 M_K(x,z)dz \, g_1^2(y_1)g_2^2(y_2) \notag \\
& + f^2(x)\ \big( g_1(y_1+\frac{1}{K_1})-g_1(y_1) \big)^2\ g_2^2(y_2)\ (K_1 \gamma(x) + b_1(x)) \ K_1 y_1 \notag \\
& +     f^2(x)\ g_1^2(y_1)\ \big (g_2(y_2+\frac{1}{K_2})-g_2(y_2) \big)^2\ (K_2 \gamma(x) + b_2(x)) \ K_2 y_2 \notag \\
& + f^2(x)\ \big( g_1(y_1-\frac{1}{K_1})-g_1(y_1) \big)^2\ g_2^2(y_2)\
\Big(K_1 \gamma(x) + d_{1}(x)+\beta_{1}(x)(y_1\lambda_{11} +y_2\lambda_{12} )\Big) K_1 y_1 \notag \\
& + f^2(x)\ g_1^2(y_1)\ \big( g_2(y_2-\frac{1}{K_2})-g_2(y_2) \big)^2\
 \Big( K_2 \gamma(x) + d_{2}(x)+\beta_{2}(x)(y_1\lambda_{21} + y_2\lambda_{22} )\Big) K_2 y_2
 \bigg\} \notag \\
& \hskip 9.3cm Z^\kappa_s(dx,dy_1,dy_2)\ ds. \label{qv3}
\end{align}
\end{description}
\end{prop}

\noindent We assume that the sequence of initial conditions
$Z_0^\kappa$ converges in law to some finite measure
$\zeta_0$. Let us study the limiting behavior of the processes $Z^\kappa$
as $\kappa$ tends to infinity. It depends on the value of $\eta$ and 
leads to two different convergence results.\\
As before we denote by  $r_{i}$ the rate $ b_{i}-d_{i}$. 

\newpage
\begin{thm}
\label{largepopcellsbirthratedet}
Assume $(H3)$ and $\eta \in ]0,1[$; suppose that the initial conditions $Z^\kappa_0 \in M_F({\cal X} \times \rit_+^2)$ satisfies 
$\ \sup_\kappa E(\langle Z^\kappa_0,1\rangle^3)<+\infty$.
If  further, the sequence of measures $(Z^\kappa_0)_\kappa$  converges in law 
to a finite deterministic measure
$w_0$, then the sequence of processes $(Z^\kappa_t)_{0\leq t\leq T}$ converges in
law in the Skorohod space $\dit([0,T],M_F({\cal X} \times \rit_+^2))$, as $\kappa$ goes
to infinity, to the unique (deterministic) flow of functions $w \in C([0,T], \mathbb{L}^1({\cal X} \times \mathbb{R}_{+}^2))$  weak solution of
\be
\label{reac-dif}
   {\partial \over \partial t} w_{t} = \Big(B - D - \alpha\ U*w_t \Big)w_{t} 
+ \bigtriangleup_{x}\left(p\, \sigma^2\, \Gamma\, w_{t}\right)  +  \,\bigtriangleup_{y} (\gamma \,w_{t}) 
  -  \triangledown_{y} \cdot (c w_t).
   \ee
\end{thm}

\begin{rema}   One obtains the existence and uniqueness of function-valued solutions of \eqref{reac-dif} even if the initial measure $w_0$ is a degenerate  one without density.
\end{rema}

\begin{proof}
The proof follows the same steps as the one of Theorem \ref{largepoplargecells} except that the  mutation term will lead  to a Laplacian term in $f$ since the mutation kernel is centered and the mutation steps converge to $0$ in the appropriate scale. We first obtain the tightness of  the sequence $(Z^\kappa)$ and the fact that each subsequence converges to a measure-valued flow  $w \in C([0,T], M_{F}({\cal X} \times \mathbb{R}_{+}^2))$ satisfying for bounded $C^2$-functions $f, g_{1}, g_{2}$, 
\begin{align}
\label{eq:diffusion}
& \langle w_t,fg_1g_2\rangle  =\langle w_0,fg_1g_2\rangle  +\int_0^t\int_{{\cal  X}\times \rit_+^2} \bigg\{\Big( B(x,y_1,y_2)\notag \\
& \hskip 2cm - D(x,y_1,y_2)- \alpha(x,y_1,y_2)\ U*w_s(x,y_1,y_2)\Big)
  f(x) g_1(y_1)g_2(y_2) \notag \\
& \qquad + p(x)\sigma^2(x) \Gamma(x,y_1,y_2)  \bigtriangleup f(x)  \, g_1(y_1)g_2(y_2) \notag \\
& \qquad + f(x) \triangledown g_1(y_1) g_2(y_2)
\Big( r_{1}(x) - \beta_{1}(x)(y_1\lambda_{11}+y_2\lambda_{12})\Big)\ y_1  \notag \\
& \qquad + f(x)  g_1(y_1) \triangledown g_2(y_2) \Big( r_{2}(x) - \beta_{2}(x)(y_1\lambda_{21}+y_2\lambda_{22})\Big)\ y_2 \notag \\
& \qquad +  f(x)  \gamma(x)\left( \bigtriangleup g_1(y_1) g_2(y_2) 
 +g_1(y_1) \bigtriangleup g_2(y_2) \right) \bigg\}  w_s(dx,dy_1,y_2)\, ds.
\end{align}
We can also apply $w_t$ to smooth time-dependent test functions $h(t,x,y)$  defined on $\mathbb{R}_{+} \times {\cal X}\times \mathbb{R}_{+}^2$
%,  of class $C^1$ on $\mathbb{R}_{+}$ and $C^2$ on ${\cal X}\times (\mathbb{R}_{+})^2$. 
. That will add a term of the form $\langle \partial_{s}h,w_{s}\rangle$ in \eqref{eq:diffusion}.  

\noindent Let us now sketch the uniqueness argument.
Thanks to the Lipschitz continuity and ellipticity assumption $(H3)$,  the semigroup associated with the infinitesimal generator
   $$
{\cal A}:= p\, \sigma^2\, \Gamma \bigtriangleup_{x}  +  \,\gamma \bigtriangleup_{y} 
  +  \, c \cdot \triangledown_{y}
$$  
admits at each time $t>0$ a smooth  density denoted by $\psi^{x,y}(t,\cdot,\cdot) $ on ${\cal  X} \times \mathbb{R}_+^2$. That is, for any bounded continuous function $G$ on  ${\cal X}\times \mathbb{R}_{+}^2$, the function $$
\check G(t,x,y)= \int  \psi^{x,y}(t,x',y') G(x',y')dx'dy'
$$ 
 satisfies 
$$
   {\partial \over \partial t}\check G = {\cal A} \check G; \quad  \check G(0,\cdot,\cdot) = G.
   $$

Thus  \eqref{eq:diffusion}  applied to the test function  $(s,x,y) \mapsto \check G(t-s,x,y)$ leads to the mild equation: for any continuous and bounded function $G$,
 \begin{align}
\label{eq:milddiffusion}
\langle w_t,G\rangle & =\langle w_0,\check G(t,\cdot)\rangle 
 +\int_0^t \langle w_s, (B-D-\alpha U*w_s) 
  \check G(t-s,\cdot )\rangle
     \  ds\notag\\
& =\int_{{\cal X}\times \rit_+^2}  G(x', y')  \int_{{\cal X}\times \rit_+^2}  \psi^{x,y}(t,x',y') w_0(dx,dy) \ dx'dy'  \notag \\
& + \int G(x', y') \int_0^t 
  \int \big(B-D-\alpha U*w_s\big)(x,y)\, \psi^{x,y}(t-s,x',y')  
     w_s(dx,dy)\  ds \ dx'dy'.
\end{align}  
It is simple to deduce from this representation the uniqueness of the measure-valued solutions of (\ref{eq:diffusion}). Moreover, by Fubini's theorem and $(H3)$ and since $\sup_{t\leq T} \langle w_{t},1\rangle <+\infty$, one observes that 
$$\langle w_t,G\rangle  =\int_{{\cal X}\times \rit_+^2}  G(x', y') H_{t}(x',y') dx' dy',$$
with $H\in \mathbb{L}^\infty([0,T], \mathbb{L}^1({\cal X}\times \rit_+^2))$. Thus for any $t\leq T$, the finite measure $w_{t}$ is absolutely continuous with respect to the Lebesgue's measure and the solution of (\ref{eq:diffusion}) is indeed a function for any positive time.
      \end{proof}      
    
\bigskip

\noindent If $\eta=1$ the limiting process of $Z^\kappa$ is no more deterministic but is a random superprocess with values in $ C([0,T], M_{F}({\cal X} \times \mathbb{R}_{+}^2))$.
 
\begin{thm}
\label{largepopcellsbirthrate}
Assume $(H3)$ and $\eta =1$. Assume moreover that the initial
conditions $Z^\kappa_0 \in M_F({\cal X} \times \rit_+^2)$ satisfy $\ \sup_\kappa E(\langle Z^\kappa_0,1\rangle^3)<+\infty$. If they converge in law 
as $\kappa$ tends to infinity to a finite deterministic measure
$\zeta_0$, then the sequence of processes $(Z^\kappa_t)_{0\leq t\leq T}$ converges in
law in the Skorohod space $\dit([0,T],M_F({\cal X} \times \rit_+^2))$, as $\kappa$ goes
to infinity, to the continuous measure-valued semimartingale $\zeta \in
C([0,T],M_F({\cal X} \times \rit_+^2))$ satisfying  for any bounded smooth functions $f, g_1,g_2$:
\begin{align}
\label{eq:limit3}
& M_t^{fg}:= \langle \zeta_t,fg_1g_2\rangle  - \langle \zeta_0,fg_1g_2\rangle \notag \\
& - \int_0^t\int_{{\cal X} \times \rit_+^2} \bigg\{\Big( B(x,y_1,y_2) - D(x,y_1,y_2)
- \alpha(x,y_1,y_2) \ U*\zeta_s(x,y_1,y_2) \Big)    f(x) g_1(y_1)g_2(y_2) \notag \\
& \qquad + p(x)\sigma^2(x) \Gamma(x,y_1,y_2) \bigtriangleup f(x) g_1(y_1)g_2(y_2)  \notag \\
& \qquad + f(x) \triangledown g_1(y_1) g_2(y_2)
\Big( r_{1}(x)- \beta_{1}(x)(y_1\lambda_{11}+y_2\lambda_{12})\Big)\ y_1  \notag \\
& \qquad + f(x)  g_1(y_1) \triangledown g_2(y_2) \Big( r_{2}(x) - \beta_{2}(x)(y_1\lambda_{21}+y_2\lambda_{22})\Big)\ y_2 \notag \\
& \qquad + f(x)\gamma(x)
\left( \bigtriangleup g_1(y_1) g_2(y_2)   +  g_1(y_1) \bigtriangleup g_2(y_2)\right) \bigg\} \zeta_s(dx,dy_1,dy_2)\  ds
\end{align}
is a  continuous square integrable $({\cal F}_t)_{t\geq 0}$-martingale 
with quadratic variation
\be
\label{qv3}
\langle M^{fg}\rangle_t = \intot \int_{ {\cal X} \times \rit^2_+}
2 \, \Gamma(x,y_1,y_2) f^2(x)g_1^2(y_1)g_2^2(y_2)
\zeta_s(dx,dy_1,dy_2)\ ds.\ee
\end{thm}

\begin{proof}   The convergence is obtained by a compactness-uniqueness  argument. 
The uniform tightness of the laws and the identification of the limiting values can be adapted 
from  \cite{FM04} with some careful moment estimates and an additional drift term  as in the proof of Theorem \ref{largepoplargecells}.

\noindent The uniqueness can be deduced from the one with $B=D= \alpha=0$ by using the Dawson-Girsanov transform for measure-valued processes
 (cf. Theorem~2.3 in \cite{EP94}), as soon as  the ellipticity assumption for $\Gamma$  is satisfied.  Indeed,  
\begin{equation*}
\mathbb{E}\left( \int_0^t\int_{{\cal X} \times \rit_+^2} 
 \Big( B(x,y_1,y_2) - D(x,y_1,y_2)
- \alpha(x,y_1,y_2) \ U*\zeta_s(x,y_1,y_2) \Big)^2 \zeta_s(dx,dy_1,dy_2) ds\right)<+\infty,
\end{equation*}
which allows us to use this transform. 

\medskip
\noindent  In the case $B=D=\alpha=0$ the proof of  uniqueness can be  adapted from the general results of Fitzsimmons, see \cite{F92} Corollary 2.23: the Laplace transform of the process is uniquely identified using the extension of the  martingale problem  \eqref{eq:limit3} to test functions  depending smoothly on the time like $(s,x,y_{1},y_{2})\mapsto \psi_{t-s}f(x,y_{1},y_{2})$ for bounded functions $f$ (see \cite{F92} Proposition 2.13). 

\end{proof}

\end{document}